\documentclass[11pt]{amsart}
\usepackage{enumerate}
\usepackage{hyperref}
\usepackage{amsthm,amsfonts,amsmath,amscd,amssymb}
\usepackage{xcolor}
\usepackage{graphicx}
\usepackage{caption}
\usepackage{subcaption}
\let\oldmarginpar\marginpar
\renewcommand\marginpar[1]{\-\oldmarginpar[\raggedleft\footnotesize #1]%
{\raggedright\footnotesize #1}}
\theoremstyle{plain}
\newtheorem{thm}{Theorem}[section]
\newtheorem{lemma}[thm]{Lemma}
\newtheorem{prop}[thm]{Proposition}
\newtheorem{coro}[thm]{Corollary}

\theoremstyle{definition}

\newtheorem{definition}[thm]{Definition}
\newtheorem{remark}[thm]{Remark}

\theoremstyle{remark}

\numberwithin{equation}{section}

\renewcommand{\a}{\alpha}
\renewcommand{\S}{\mathbb{S}}
\newcommand{\D}{\mathbb{D}}

\newcommand{\Z}{\mathbb{Z}}

\newcommand{\R}{\mathbb{R}}
\newcommand{\C}{\mathbb{C}}

\newcommand{\J}{\mathcal{J}}
\newcommand{\La}{\Lambda}
\newcommand{\la}{\lambda}
\newcommand{\p}{\varphi}
\newcommand{\e}{\varepsilon}
\newcommand{\dd}{\partial}

\newcommand{\op}{\operatorname}

\newcommand{\sse}{\subseteq}
\newcommand{\les}{\leqslant}

\newcommand{\x}{\times}
\newcommand{\sm}{\setminus}
\newcommand{\pt}{\text{point}}
\newcommand{\const}{\text{constant}}

\newcommand{\Symp}{\operatorname{Symp}}

\newcommand{\wt}{\widetilde}

\newcommand{\ol}{\overline}

\newcommand{\std}{\text{st}}
\newcommand{\st}{\text{st}}
\newcommand{\ot}{\text{ot}}
\newcommand{\cyl}{\text{cyl}}
\newcommand{\univ}{\text{univ}}
\newcommand{\Int}{\operatorname{Int}}
\def\Op{{\mathcal O}{\it p}\,}
\begin{document}
\begin{abstract}
In this article we establish efficient geometric criteria to decide whether a contact manifold is overtwisted. Starting with the original definition, we first relate overtwisted disks in different dimensions and show that a manifold is overtwisted if and only if the Legendrian unknot admits a loose chart. Then we characterize overtwistedness in terms of the monodromy of open book decompositions and contact surgeries. Finally, we provide several applications of these geometric criteria.
\end{abstract}

\title{Geometric criteria for overtwistedness}
\subjclass[2010]{Primary: 57R17. Secondary: 53D10, 53D15.}

\author{Roger Casals}
\address{University of California Davis, Department of Mathematics, Shields Avenue, Davis, CA 95616, United States of America}
\email{casals@math.ucdavis.edu}

\author{Emmy Murphy}
\address{Northwestern University, Department of Mathematics, 2033 Sheridan Road Evanston, IL 60208, United States of America}
\email{e\_murphy@math.northwestern.edu}

\author{Francisco Presas}
\address{Instituto de Ciencias Matem\'aticas CSIC, C. Nicol\'as Cabrera 13 28049 Madrid, Spain}
\email{fpresas@icmat.es}

\maketitle
\vspace{-1cm}
\section{Introduction} \label{sec: intro}

Symplectic and contact topology intertwine the global behaviour from differential topology with subtle rigid geometric structures \cite{ArGi,EGH,ElFR,Gr}. Both sides of the flexible--rigid dichotomy \cite{ElFR} feature prominently in the field; the present work belongs to the flexible side of contact topology. The main result in this article is Theorem \ref{thm:main}, which characterizes the contact structures satisfying the parametric $h$--principle \cite{BEM,PDR} in terms of the different geometric notions existing in the literature in contact topology, including adapted open book decompositions \cite{Co2,Gi2}, contact surgeries {\cite{ElStein,We91}, loose Legendrian submanifolds \cite{loose} and obstructions to fillability \cite{PS}. This work proves that the existence of these objects with additional properties, which conjecturally led to an $h$--principle, does indeed imply that the $h$--principle is satisfied.

The geometric criteria stated in Theorem \ref{thm:main} can be verified in several interesting cases, and in particular we provide the first explicit examples of overtwisted contact manifolds in higher dimensions. Theorem \ref{thm:main} has been used in a variety of contexts, such as \cite{CM,ElMu,Huang}, and we strongly believe it captures the most efficient ways to detect overtwistedness. The central goal of this article is to prove Theorem \ref{thm:main}.

\subsection{The main theorem}\label{ssec:main}
A contact structure on a ($2n-1$)--dimensional smooth manifold $Y$ is a maximally non--integrable hyperplane distribution $\xi\sse TY$, a succinct introduction to the properties of such hyperplane distributions can be found in \cite[Chapter 4]{ArGi}. There exists a remarkable class of contact structures, which has been introduced in \cite[Definition 3.6]{BEM} in any dimension, called the {\it overtwisted} contact structures. Generalizing the original definition and results in the $3$-dimensional case \cite{El}, it is shown in \cite[Theorem 1.2]{BEM} that overtwisted contact structures satisfy a parametric $h$-principle \cite{ElMi}, i.e.~their classification up to contact isotopy coincides with the classification of homotopy classes of almost contact structures. This classification then becomes a strictly algebraic topological problem which can be solved via obstruction theory \cite{Ha}. The definition of the class of overtwisted contact structures provided in \cite[Definition 3.6]{BEM} will be reviewed in Section \ref{sec:OTxD2}, but for the time being the reader can think of them as contact structures containing a certain contact hypersurface germ in the same vein that the 3--dimensional case \cite{El}.

The result in the article \cite[Theorem 1.1]{BEM} does demonstrate the existence of overtwisted contact structures homotopic to any almost contact structure, but a crucial drawback to the existence proof is that the construction is not explicit. In consequence, there were no explicit examples of closed overtwisted contact manifolds of dimension $2n-1\geq5$, and the techniques used in \cite{BEM} give no criterion in order to show that a given manifold is overtwisted, other than a direct application of the definition, which to the knowledge of the authors has never been done. The geometric criteria Theorem \ref{thm:main} entirely solves these problems by providing a number of equivalent conditions which characterize overtwistedness. In addition, it brings together different geometric objects used to study contact structures, thus establishing a unifying ground for flexibility across the different facets of contact topology.

Let us now state our main result. The notions and notations used in the following statement will be explained in the rest of this Subsection \ref{ssec:main}.

\begin{thm}\label{thm:main}
Let $(Y, \xi)$ be a contact manifold of dimension $2n-1\geq5$ and $\alpha_\ot\in\Omega^1(\R^3)$ a $1$--form such that $(\R^3,\ker\a_\ot)$ is an overtwisted contact structure.

Then the following conditions are equivalent:
\begin{enumerate}[(i)]
\item[1.] The contact structure $(Y, \xi)$ is overtwisted.
\vspace{0.1cm}
\item[2.] There is a contact embedding of $(\R^3 \x \C^{n-2}, \ker\{\alpha_\ot + \lambda_\std\})$ into $(Y, \xi)$.
\vspace{0.1cm}
\item[3.] The standard Legendrian unknot $\Lambda_0 \sse (Y,\xi)$ is a loose Legendrian submanifold.
\vspace{0.1cm}
\item[4.] $(Y, \xi)$ contains a small plastikstufe with spherical core and trivial rotation.
\vspace{0.1cm}
\item[5.] There exists a contact manifold $(Y', \xi')$ and a loose Legendrian submanifold $\Lambda \sse (Y',\xi')$ such that $(Y,\xi)$ is contactomorphic to the contact $(+1)$--surgery of $(Y',\xi')$ along $\Lambda$.
\vspace{0.1cm}
\item[6.] There exists a negatively stabilized contact open book compatible with $(Y, \xi)$.\hfill$\Box$
\end{enumerate}
\end{thm}
We will momentarily discuss the different items in the statement of Theorem \ref{thm:main}. However, we first note that the second condition in the geometric criteria is actually verifiable.
\begin{remark}\label{rmk:2a=2b}
There exists a positive constant $R=R(n,\a_\ot)\in\R^+$, which only depends on the choice of overtwisted contact form $\alpha_\ot$ and the dimension of the contact manifold $(Y,\xi)$ such that the second condition in Theorem \ref{thm:main} on the existence of the contact embedding of
$$(N_\infty,\xi_\infty)=(\R^3 \x \C^{n-2}, \ker\{\alpha_\ot + \lambda_\std\})$$
into the contact manifold $(Y,\xi)$ is equivalent to the existence of a contact embedding of
$$(N_R,\xi_R)=(\R^3 \x D^{2n-4}(R), \ker\{\alpha_\ot + \lambda_\std\})$$
into the contact manifold $(Y, \xi)$, as it follows from the h--principle \cite[Corollary 1.4]{BEM}. This critical radius $R(n,\a_\ot)$ is a finite number, and thus in order to verify overtwistedness of the contact manifold $(Y,\xi)$ using the second characterization in Theorem \ref{thm:main} it suffices to find contact embeddings of $(N_\rho,\xi_\rho)$ into $(Y,\xi)$ for any finite radius $\rho\in\R^+$. Then choosing $\rho\in\R^+$ such that $R(n,\a_\ot)\leq\rho$ yields the existence of a contact embedding of the infinite radius contact domain $(N_\infty,\xi_\infty)$ into $(Y,\xi)$.\hfill$\Box$
\end{remark}

Let us now describe the elements entering in the statement of Theorem \ref{thm:main} above in more detail. The 1--form $\lambda_\std\in\Omega^1(\C^{n-2})$ denotes the standard Liouville form on both the open disk $D^{2n-4}(R)$ and the complex Euclidean space $\C^{n-2}$, which in the standard coordinates $(x_1,y_1,\ldots,x_{n-2},y_{n-2})=(r_1,\theta_1,\ldots,r_{n-2},\theta_{n-2})$ is expressed as
$$\lambda_\std = \frac12\sum_{i=1}^{n-2}(x_idy_i - y_idx_i)=\frac12\sum_{i=1}^{n-2}r_i^2d\theta_i.$$
There are infinitely many choices for an overtwisted contact form $\a_\ot$ for the real Euclidean space $\R^3(z,r,\theta)$, the most used in the literature is the rotationally symmetric $z$--invariant form
$$\a_\ot=\cos(r)dz+r\sin(r)d\theta,$$
but note that the statement of Theorem \ref{thm:main} only requires a contact form $\a_\ot$ defining an overtwisted contact structure on $\R^3$, not necessarily contactomorphic to $\ker(\cos(r)dz+r\sin(r)d\theta)$. It is interesting to observe that the classification of overtwisted contact structures in $\R^3$ is understood \cite{El93}, and thus such choices can be readily classified as well.

\subsection{The Statement of Theorem \ref{thm:main}}

Theorem \ref{thm:main} proves the equivalence between different notions of flexibility in contact topology, and the several geometric objects that appear in the statement of Theorem \ref{thm:main} have been introduced by many authors in different works, who we now credit.

The definition of a higher--dimensional overtwisted disk featuring in the first item $(1)$ initially appears in the article \cite[Definition 3.6]{BEM}. The 3-dimensional overtwisted disk was initially introduced by Y.~Eliashberg in \cite[Section 1.4]{El}.

The second item $(2)$ in Theorem \ref{thm:main} features the 3-dimensional contact structure $(\R^3,\ker{\alpha_\ot})$, as discussed above, and the contact product $(\R^3 \x \C^{n-2}, \ker\{\alpha_\ot + \lambda_\std\})$. This latter higher-dimensional contact manifold should be seen as an infinitely large neighborhood of the contact submanifold $\R^3\times\{0\}$. The first breakthrough in the study of contact neighborhoods is contained in the article \cite{NP}, particularly \cite[Corollary 13]{NP}, where it is proven that a large enough neighborhood of a 3-dimensional overtwisted contact manifold contains a generalized plastikstufe.

The article \cite{NP} has been of central importance for higher-dimensional contact geometry, and the statement of the equivalence $(1)\Longleftrightarrow(2)$ has its roots in \cite{NP}. In particular, the first-named author is grateful to K.~Niederkr\"uger for several discussions on \cite{NP} and related topics. The third-named author is also thankful to him for many useful conversations.

\begin{remark}
The equivalence $(1)\Longleftrightarrow(2)$ in Theorem \ref{thm:main} is crucial in the present proof of Theorem \ref{thm:main}. Indeed, all other equivalences with the first item $(1)$ use the equivalence $(1)\Longleftrightarrow(2)$, which we shall prove first in Section \ref{sec:OTxD2}. In this sense, the equivalence $(1)\Longleftrightarrow(2)$ is the core of Theorem \ref{thm:main}. \hfill$\Box$
\end{remark}

Regarding the third item $(3)$, the notion of a loose Legendrian submanifold was introduced in the article \cite[Definition 4.3]{loose}, and the definition of a plastikstufe, appearing in the fourth item $(4)$, was given by K.~Niederkr\"uger in the article \cite[Section 1]{PS}. The concept of $(+1)$--surgery in the fifth item (5) is first detailed in the articles \cite{ElStein,We91} and the notions of a compatible open book and a negative stabilization appearing in the sixth item (6) were introduced by E.~Giroux \cite{Co2,Gi2}.

In short, we now describe and contextualize these geometric concepts before delving into their more technical nature in Section \ref{background}.

The \emph{standard Legendrian unknot} $\Lambda_0 \sse (\R^{2n-1},\xi_\std)$ is defined to be the Legendrian sphere
$$\Lambda_0 = \{y_i = 0: i=1,\ldots, n\} \cap S^{2n-1} \sse (S^{2n-1},\xi_\std) \sm \{\pt\} \sse \C^n[x_1,y_1,\ldots,x_n,y_n],$$
where the standard contact structure $(\S^{2n-1},\xi_\std)$ is defined by restriction of the Liouville form $\lambda_\st$ to the unit sphere $\S^{2n-1}=\{(x_1,y_1,\ldots,x_n,y_n)\in\C^n:\|x\|^2+\|y\|^2=1\}\sse\C^n$ and we have used the natural contactomorphism $(\R^{2n-1},\xi_\std) \cong (S^{2n-1},\xi_\std) \sm \{\pt\}$, detailed for instance in \cite[Prop.~2.1.8]{Ge}. Then the standard Legendrian unknot $\Lambda_0 \sse (Y, \xi)$ is defined by the inclusion of a Darboux chart $\Lambda_0 \sse (\R^{2n-1},\xi_\std) \sse (Y, \xi)$, all of which are isotopic.

The concept of a \emph{loose} Legendrian submanifold, which appears in the third item (3) is first studied in the article \cite{loose}, and the reader might also be interested in \cite{CM,CM2,CE,MNPS}. The fact that a Legendrian submanifold is loose is characterized by the existence of a certain piece that the Legendrian might or might not have: Theorem \ref{thm:main} states that if the most basic Legendrian, the Legendrian unknot, already contains such piece, then the ambient manifold $(Y,\xi)$ is overtwisted. Thus, we are relating the flexibility $h$--principle exhibited by loose Legendrian submanifolds \cite{loose} with the $h$--principle satisfied by overtwisted contact structures \cite[Theorem 1.2]{BEM}. Note that with the techniques developed in \cite{CM,CM2} it is much simpler to verify that the Legendrian unknot is loose than proving overtwistedness by using the definition.

The \emph{plastikstufe}, appearing in the fourth item (4), is an $n$-dimensional smooth submanifold $\mathcal{P} \sse (Y, \xi)$ such that the germ of the ambient contact structure $\xi$ in an open neighborhood of $\mathcal{P}$ is given by an explicit local model, inspired by the definition of the overtwisted disk \cite{El,Gr} in the three--dimensional case. The plastikstufe was first defined in the article \cite{PS} and shown to be an obstruction to symplectic fillability in higher--dimensions, in the same manner that the overtwisted $2$--disk in a contact 3--fold obstructs the existence of a 4--dimensional symplectic filling. The technical definitions of {\it small} and {\it trivial rotation} were first introduced in \cite{MNPS}, where the existence of a plastikstufe is studied in relation to loose charts for Legendrians submanifolds. Both hypotheses are technical, and according to the recent work \cite{Huang} they can actually be removed. The statement of Theorem \ref{thm:main} concerning the plastikstufe thus relates the existence of an explicit $n$-dimensional contact germ $\mathcal{P}$, built as a family of overtwisted $2$--disks, with higher--dimensional overtwistedness \cite{BEM}. This has meaningful advantages, such as the fact that there are simple geometric constructions of plastikstufes \cite{Ca,CPS,NP} and it is often simpler to find the contact germ $\mathcal{P}$ than the higher--dimensional overtwisted contact germ \cite[Definition 3.6]{BEM}.

In the fifth item of Theorem \ref{thm:main}, overtwisted contact manifolds are characterized as the contact manifolds which admit a contact surgery presentation in which one of the Legendrian $(+1)$--components of the surgery link admits a loose chart in the complement of the other components. The definition of a {\it contact $(+1)$--surgery} along a Legendrian sphere is implicit in the theory of Weinstein handle attachments \cite{CE,ElStein,We91}; it is the surgery induced by attaching a \emph{concave} handle to a compact piece of the symplectization. In the higher dimensional case it was studied in more depth in \cite{Avdek}, where the implication (5)$\Rightarrow$(1) proven in Theorem \ref{thm:main} is stated as Conjecture 9.16. The essential fact that the reader should remember regarding this fifth item is that we prove that contact $(+1)$--surgery along a loose Legendrian sphere always yields an overtwisted contact manifold. The study of contact structures from the surgery viewpoint is well--understood in the three--dimensional case \cite{OS} and it is currently a developing field of interest in higher--dimensions \cite{CM,CM2}.

Finally, {\it compatible open books} appear as the sixth geometric criteria to detect overtwistedness. In order for this characterization in Theorem \ref{thm:main} to apply we also suppose that $(Y,\xi)$ is a closed contact manifold. The notion of an open book compatible with a contact structure $(Y, \xi)$ was first introduced by E.~Giroux in his study of the correspondance between open books and contact structures \cite{Co2,Gi2}. In brief, it states that an appropriate open book decomposition of the smooth manifold $Y$ determines a contact structure $\xi$, and conversely every contact manifold $(Y,\xi)$ admits such an adapted open book decomposition.

The open books compatible with a contact structure $(Y,\xi)$ admit a contact operation: they can be positively stabilized or negatively stabilized. The resulting open books induce two contact structures $\xi_+$ and $\xi_-$ on the smooth manifold $Y$. The positive stabilization $(Y,\xi_+)$ is an operation which yields a contact structure contactomorphic to $(Y,\xi)$, but the contact structure $(Y,\xi_-)$ resulting from a negative stabilization is oftentimes not contactomorphic to $(Y,\xi)$. In particular, the negative stabilization of a contact structure is known to have vanishing symplectic field theory \cite{BN,BvK}. In particular, using \cite[Theorem 1.3]{BvK}, Theorem \ref{thm:main} implies that the contact homology of an overtwisted contact manifold vanishes.

These geometric objects will be discussed in more technical detail in the subsequent Section \ref{background}, but we hope that the above description provides some context for Theorem \ref{thm:main} and it helps the reader to navigate between the diverse range of objects in its statement.
\subsection{The argument for Theorem \ref{thm:main}}\label{ssec:thmmain} First, there are six equivalences stated in Theorem \ref{thm:main} and there is by no means a canonical approach nor a natural order to prove them. However, we have chosen a route that in our perspective most enlightens the connection between the different geometric objects and also minimizes the need for a thorough understanding of the article \cite{BEM}.

To begin with, the argument we use to prove Theorem \ref{thm:main} is in its entirety an induction in the dimension $(2n-1)$. That is, we shall first prove Theorem \ref{thm:main} for contact 5--folds $(Y,\xi)$, which constitutes the base case, and we will then show that if the statement is true for any smooth manifold $Y$ with $\dim(Y)=2n-3$ then it is also true for $(2n-1)$--dimensional manifolds. With this in mind, the equivalences will be proven according to the following program:

\begin{itemize}
\item[-] The equivalence (1)$\Longleftrightarrow$(2) is the content of Theorem \ref{thm:OTxD2}, proved in Section \ref{sec:OTxD2}.\\

\item[-] The equivalences (1)$\Longleftrightarrow$(3)$\Longleftrightarrow$(4): the implication (3)$\Rightarrow$(1) is proven in Section \ref{sec: cobord} as a consequence of Theorem \ref{thm: loose to OT}. The main ingredient is Lemma \ref{lem:cobord}, which is where the inductive hypothesis is used. Note that (4)$\Rightarrow$(3) follows from \cite[Theorem 1.1]{MNPS}.\\

\item[-] The equivalence (1)$\Longleftrightarrow$(5) is proven in Section \ref{sec: surgery} where we show the implication (5)$\Rightarrow$(4).\\

\item[-] The equivalence (1)$\Longleftrightarrow$(6) is shown in Section \ref{sec:NegStab}, with an argument proving (6)$\Rightarrow$(3).
\end{itemize}

The implications we have emphasized above are the ones that require new ideas and techniques. The remaining implications needed in order to obtain the equivalences follow from the $h$--principle: the relative parametric $h$--principle, \cite[Theorem 1.1]{BEM} and \cite[Theorem 1.2]{BEM}, does imply (1)$\Rightarrow$(2), (1)$\Rightarrow$(3) and (1)$\Rightarrow$(4), and the implications (1)$\Rightarrow$(5) and (1)$\Rightarrow$(6) are not hard. The real effort, as in any result characterizing an $h$--principle, is to prove the converse implications by constructing an overtwisted disk from a priori weaker geometric object. Section \ref{sec:mainproof} contains the proof of Theorem \ref{thm:main} gathering the equivalences above.

\begin{remark}
Here is an alternative route that two of the authors have also used in talks since it minimizes the use of the $h$--principle \cite[Theorem 1.1]{BEM}. First, one proves the equivalence (1)$\Longleftrightarrow$(2) with the argument in this article, and then proceeds with the following sequence:

\begin{itemize}
\item[-] The equivalence (3)$\Longleftrightarrow$(6) can be proven directly with the techniques we develop in Section \ref{sec: cobord}. This is a self--contained relation.
\item[-] The equivalence (2)$\Longleftrightarrow$(3) then can be established by proving (3)$\Rightarrow$(2) from our cobordism argument in Section \ref{sec: cobord}, and deducing the implication (2)$\Rightarrow$(3) by adapting the classical 3--dimensional destabilizing argument in the presence of an overtwisted disk \cite{MNPS}.
\item[-] The implications (5)$\Rightarrow$(2) and (5)$\Rightarrow$(3) can be proven directly by studying Weinstein handle attachment in detail \cite{CM,CM2,ElMu}, and finally the implication (6)$\Rightarrow$(5) follows from the fact that the zero section in $\op{OB}(T^*S^n,\op{id})$ is a loose Legendrian submanifold \cite{CM2}.
\end{itemize}
Hence the equivalences (2)$\Longleftrightarrow$(3)$\Longleftrightarrow$(4)$\Longleftrightarrow$(5)$\Longleftrightarrow$(6) do not require the overtwisted $h$--principle \cite[Theorem 1.2]{BEM}. Nevertheless, they require the loose Legendrian $h$--principle \cite{loose} and the main arguments in this article. Thus, from a flexible perspective it is neater to directly use the $h$--principle \cite[Theorem 1.2]{BEM} to immediately conclude the converses, which also explains our choice of strategy.\hfill$\Box$
\end{remark}
\subsection{Organization} The article contains eight sections, which we have distributed as follows. First, Section \ref{background} provides the required background in contact topology in order to follow the article. Then, Sections \ref{sec:OTxD2}, \ref{sec: cobord}, \ref{sec: surgery} and \ref{sec:NegStab} contain the main results for the proof of Theorem \ref{thm:main}, these results are divided in terms of the equivalences they are used to prove in Theorem \ref{thm:main}. Section \ref{sec:mainproof} contains the proof of Theorem \ref{thm:main}. Section \ref{ssec:cons} details two applications of Theorem \ref{thm:main}.

Section \ref{sec:OTxD2} proves the first equivalence (1)$\Longleftrightarrow$(2), Section \ref{sec: cobord} establishes the two equivalences (1)$\Longleftrightarrow$(3)$\Longleftrightarrow$(4), Section \ref{sec: surgery} then proves the equivalence (1)$\Longleftrightarrow$(5) and Section \ref{sec:NegStab} concludes with the proof of the equivalence (1)$\Longleftrightarrow$(6). Each of these sections also contains results that can be of interest on their own. In particular, we believe that the connection developed in Section \ref{sec:NegStab} is relevant for high--dimensional contact topology, as the subsequent work \cite{CM2} hopefully illustrates. Finally, Section \ref{ssec:cons} gives some applications of Theorem \ref{thm:main} to contact squeezing and constructions of Weinstein cobordisms.\hfill$\Box$
\subsection{Acknowledgements} We are grateful to M.S.~Borman and Y.~Eliashberg for many useful discussions. We thank O.~van Koert, O.~Plamenevskaya and K.~Siegel for valuable conversations and U.~Varolgunes and C.~Wendl for comments on the article.

F.~Presas is indebted to R.~Casals for suggesting to study the first equivalence in Theorem \ref{thm:main}. He would also like to acknowledge him for pushing this project with so much insight and determination.

R.~Casals is supported by the NSF grant DMS-1841913 and a BBVA Research Fellowship and E.~Murphy is supported by the NSF grant DMS-1510305 and a Sloan Research Fellowship. F.~Presas is supported by the Spanish Research Projects SEV--2015--0554, MTM2016--79400--P and MTM2015--72876--EXP.\hfill$\Box$
\section{Preliminaries}\label{background}

In this section we detail a number of relevant definitions and results in high--dimensional contact topology which are used along the article. In particular, we have included the definitions of the geometric objects in the statement of Theorem \ref{thm:main}. These preliminaries are necessary both for the understanding of its statement as well as its proof.

\subsection{Loose Legendrians} The notion of a loose Legendrian submanifold appears in the equivalence (1)$\Longleftrightarrow$(3), where overtwistedness is characterized in terms of the Legendrian unknot being a loose Legendrian; let us now define this class of Legendrian submanifolds. First, let $B^3\sse(\R^3,\xi_\std)$ be the round 3--dimensional ball in a contact Darboux chart and let $\Lambda_S \sse (\R^3,\xi_\std)$ be the 1--dimensional stabilized Legendrian arc depicted in Figure \ref{fig: stab}.

\begin{center}
\begin{figure}[h!]
\includegraphics[scale=0.45]{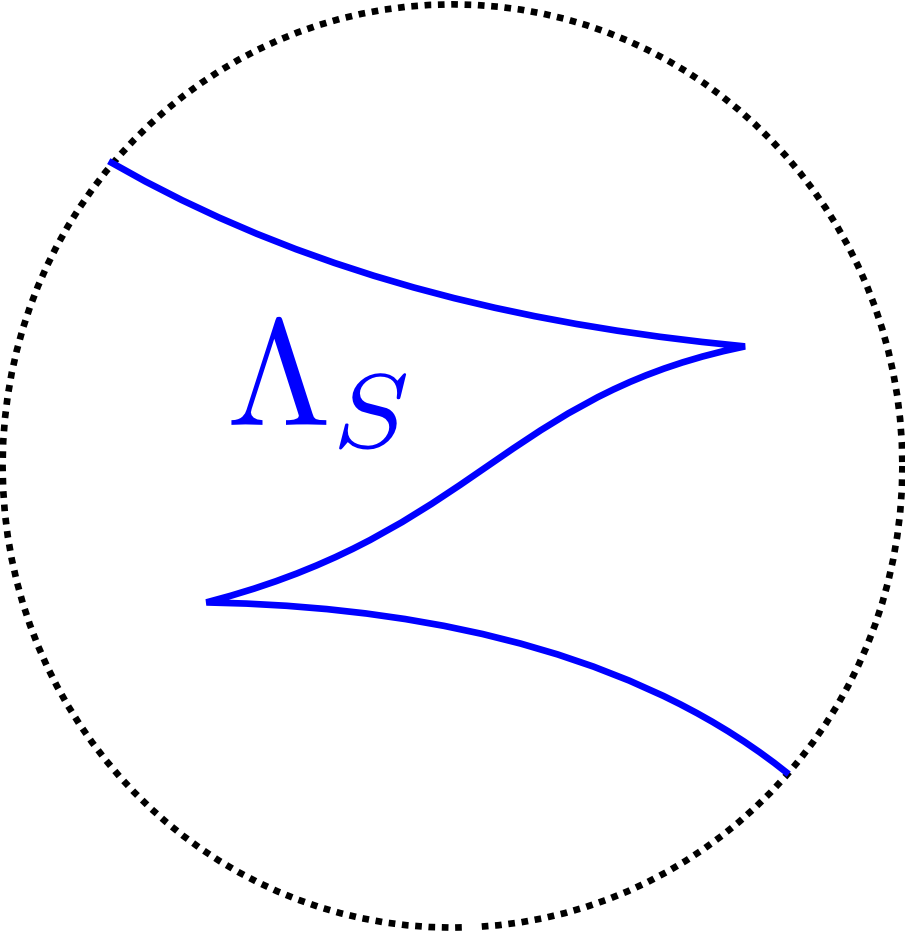}
\caption{The front projection of a stabilized Legendrian arc.}
\label{fig: stab}
\end{figure}
\end{center}

Then consider a closed manifold $Q$ and an open neighborhood $\Op(Z) \sse T^*Q$ of the zero section $Z \sse T^*Q$, and note that the product smooth submanifold
$$\Lambda_S \x Z\sse (B^3 \x \Op(Z),\ker(\alpha_\std + \lambda_\std))$$
is a Legendrian submanifold of the contact structure $\ker(\alpha_\std + \lambda_\std)$. This is the crucial local model that defines looseness, as precised in the following definition.

\begin{definition} The contact pair $(B^3 \x \Op(Z), \Lambda_S \x Z)$ endowed with the contact structure $\ker(\alpha_\std + \lambda_\std)$ is said to be a \emph{loose chart}, where $Z\sse T^*Q$ is the zero section of the cotangent bundle $(T^*Q,\la_\std)$ and $Q$ an arbitrary closed manifold.

Let $\Lambda \sse (Y, \xi)$ be a Legendrian submanifold in a contact manifold with $\dim(Y)\geq5$. The Legendrian $\Lambda$ is \emph{loose} in the contact manifold $(Y,\xi)$ if there is an open set $V \sse Y$ such that the contact pair $(V, V \cap \Lambda)$ is contactomorphic to a loose chart.\hfill$\Box$
\end{definition}

Loose Legendrians were classified up to Legendrian isotopy in the article \cite{loose}, and although the definition presented above differs slightly from the one presented in \cite{loose}, both definitions are equivalent \cite[Section 4.2]{MNPS}. The following property, which is satisfied by loose Legendrians, but not by all Legendrians, will be most useful for us. It constitutes the characterizing property of loose Legendrians, and as a basic form of an $h$--principle \cite{ElMi} it indicates that this class of Legendrians is related to the flexible side of contact topology.

{\bf
\begin{thm}[\cite{loose}]\label{thm: c0 loose}
	Let $\Lambda \longrightarrow (Y, \xi)$ be a loose Legendrian submanifold with a loose chart $U \sse (Y,\xi)$. Let $f_t:Y \longrightarrow Y$ be a smooth isotopy such that $f_0=\mbox{id}$ is the identity map, and the restriction $f_t|_{\La\cap U}=\mbox{id}|_{\La\cap U}$ is the identity map on $\Lambda \cap U$ for all $t \in [0,1]$.
	
	Then there exists a contact isotopy $g_t:Y \longrightarrow Y$ such that $g_t|_{Y \sm U}$ is $C^0$--close to $f_t|_{Y\sm U}$.
\end{thm}
}

Theorem \ref{thm: c0 loose} is used in Proposition \ref{prop: flexible ot} below, which in turn is needed in Theorem \ref{thm: ps to OT}, proving the implication (3)$\Rightarrow$(1), and in Theorem \ref{thm: weinstein exist}, one of the main applications of Theorem \ref{thm:main}.

\begin{remark}
The $h$--principle for loose Legendrian embeddings, Theorem \ref{thm: c0 loose}, does not provide an isotopy which is $C^0$-close near the loose chart $U\sse(Y,\xi)$. However, our arguments will not rely on that, and Theorem \ref{thm: c0 loose} will suffice for our purposes.\hfill$\Box$
\end{remark}

Theorem \ref{thm: c0 loose} is the result the reader should have in mind whenever a loose Legendrian submanifold appears along the article, and it should be read as the fact that loose Legendrians behave according to their smooth topology \cite{CM2,loose}.

\subsection{The plastikstufe}

The plastikstufe is a particular germ of a contact submanifold in a contact manifold, which coincides with the overtwisted 2--disk in the 3--dimensional case \cite{El,Gr,PS} . It appears in Theorem \ref{thm:main} as one of the characterizations of higher--dimensional overtwistedness, and we now provide the details on its definition, first introduced in the article \cite{PS}.

\begin{remark}
The initial purpose of the plastikstufe was to provide a higher--dimensional object which obstructs symplectic fillings, in the same manner that the existence of an overtwisted 2--disk in a contact 3--fold prevents the existence of a 4--dimensional symplectic filling \cite{Gr}. This leads to the geometric idea of considering a parametric family of overtwisted 2--disks and, with the appropriate count of dimensions of moduli spaces, the definition of the plastikstufe \cite{PS}.\hfill$\Box$
\end{remark}

Let $\Op(D^2_\ot) \sse (\R^3,\xi_\ot)$ be a contact neighborhood of an overtwisted disk for any overtwisted contact structure $\xi_\ot=\ker\alpha_\ot$.

\begin{definition}
Let $Z$ be a closed manifold and $\Op(Z) \sse T^*Z$ a neighborhood of the zero section. The contact manifold $(\Op(D^2_\ot) \x \Op(Z),\ker(\alpha_\ot + \lambda_\std))$ is said to be a \emph{plastikstufe}. The submanifold $Z\sse\Op(D^2_\ot) \x \Op(Z)$ is the \emph{core} of the plastikstufe.\hfill$\Box$
\end{definition}

The authors have provided constructions of the plastikstufe \cite{CM,CPS}, and they arise naturally in the contact divisor sum of two contact manifolds along overtwisted contact divisors.

The remarkable fact about plastikstufes is that, thanks to Theorem \ref{thm:main}, not only they serve as obstructions to symplectic fillability but actually can be used to detect overtwistedness in any dimension.

The proof that we provide holds for a large class of plastikstufes, but there are two technical hypothesis that are needed in order for the argument to work. In order to state one of these two hypotheses, we introduce the following notion.

Given a contact manifold $(Y,\xi)$ and a smooth Legendrian embedding $f:\La\longrightarrow Y$, the {\it rotation class} of the Legendrian embedding $f$, also called the rotation class of $\La$, is the homotopy class of the induced injective bundle map $Tf:T\La\longrightarrow f^*\xi$, considered as a map in the space of Lagrangian bundle monomorphisms. See \cite[Definition A.1]{loose} and \cite[Section 4]{MNPS} for a more detailed discussion on the rotation class. Equipped with this notion, the two technical hypothesis are given in the following definition.

\begin{definition}
Let $\Lambda_l \sse \Op(D^2_\ot)$ be an open leaf of the characteristic foliation of the overtwisted disk. The plastikstufe $(\Op(D^2_\ot) \x \Op(Z),\ker(\alpha_\ot + \lambda_\std))$ has \emph{trivial rotation} if the open Legendrian submanifold $\Lambda_l \x Z$ has trivial rotation class.

Also, a plastikstufe $\Op(D^2_\ot) \x \Op(Z) \sse (Y, \xi)$ is said to be \emph{small} if it is contained in a smooth ball in ambient manifold $Y$.\hfill$\Box$
\end{definition}

Note that the rotation class of the Legendrian $\Lambda_l\x Z$ is well defined since the hyperplane field $\xi$ has a unique framing on the smooth ball up to homotopy. Also, observe that in the case $Q = S^{n-2}$, a plastikstufe $\Op(D^2_\ot) \x \Op(Z) \sse (Y, \xi)$ is both small and has trivial rotation if and only if $\Lambda_l \x Z$ , which is a Legendrian annulus $[0,1] \x S^{n-2}$, can be included into a Legendrian disk. Then an open neighborhood of the union of the Legendrian disk and the plastikstufe is diffeomorphic to a smooth ball, and since a Legendrian disk has a unique framing it induces a trivial framing on its boundary collar.

The following theorem from \cite{MNPS} gives in particular the implication (4)$\Rightarrow$(3).

\begin{thm}[\cite{MNPS}]\label{thm: PS to loose}
Let $\mathcal{P} \sse (Y, \xi)$ be a small plastikstufe with spherical core and trivial rotation, and $\Lambda \sse (Y,\xi)$ a Legendrian submanifold disjoint from $\mathcal{P}$.

Then the Legendrian $\Lambda\sse(Y, \xi)$ is a loose Legendrian submanifold.
\end{thm}

\begin{remark}
The sphericity hypothesis of the core of the plastikstufe in Theorem \ref{thm:main} can be readily generalized, but being able to remove the hypothesis on its smallness requires more effort. This has been recently achieved by Y.~Huang \cite{Huang} using Theorem \ref{thm:main}.\hfill$\Box$
\end{remark}

Now that we have defined the contact geometric objects appearing in characterizations 3 and 4, we must address Weinstein structures since they have a fundamental role in the proof of the equivalence (1)$\Longleftrightarrow$(3) in Theorem \ref{thm:main}.

\subsection{Weinstein manifolds}

This subsection contains a succinct treatment on Weinstein cobordisms, where we state the results that will be used in the proof of Theorem \ref{thm:main} related to Weinstein structures. The reader is invited to study the thorough account \cite{CE} for further results on these structures.

First, the study of Weinstein structures aims at the understanding of contact and symplectic structures from the Morse theoretical viewpoint. The theory of Morse functions in smooth topology intertwines with contact and symplectic topology by requiring a compatibility condition between the Morse functions and the symplectic structure \cite{CE}. The objects of interest are the content of the following definition.

\begin{definition}
A \emph{Weinstein cobordism} is a triple $(W, \lambda, f)$, where the pair $(W, d\lambda)$ is a compact symplectic manifold with boundary, $f: W \longrightarrow [0,1]$ is a Morse function such that $\dd W = \dd_-W \cup \dd_+W = f^{-1}(0) \cup f^{-1}(1)$, and the vector field $V_\lambda$ symplectic dual to the Liouville form $\lambda$ is a gradient--like vector field for the Morse function $f$.\hfill$\Box$
\end{definition}

From the definition it follows that the 1--form $\lambda|_{f^{-1}(c)}$ is a contact form on the submanifold $f^{-1}(c)$ for any regular value $c\in[0,1]$, and note that the descending manifold $D^k_p$ associated to any critical point $p$ of $f$ satisfies $\lambda|_{D^k_p} = 0$. In particular the submanifold $D^k_p$ is isotropic and thus
$$\op{ind}(p)=k \les n = \frac12\dim W.$$
Critical points with index strictly less than $n$ are called \emph{subcritical}, and a \emph{subcritical} Weinstein cobordism $(W,\lambda,f)$ is one where all critical points of $f$ are subcritical.

In the case that $c\in[0,1]$ is a regular value, the intersection $\Lambda_p^c = D^k_p \cap f^{-1}(c)$ is an isotropic submanifold of the contact manifold $(f^{-1}(c),\ker\lambda)$. In case $c\in[0,1]$ is a critical value with a unique critical point $p \in W$, the Weinstein cobordism $(f^{-1}([c-\e, c+\e]),\lambda,f)$ is determined, up to homotopy through Weinstein structures, by the contact manifold $(f^{-1}(c-\e),\ker\lambda)$ and the (parametrized) isotropic submanifold $\Lambda^{c-\e}_p$, together with a framing of the symplectic normal bundle, which is necessarily trivial. Hence the contact manifold $(f^{-1}(c+\e),\ker\lambda)$ is determined up to contactomorphism, and it is said to be obtained from $(f^{-1}(c-\e),\ker\lambda)$ by \emph{contact surgery along} the isotropic sphere $\Lambda^{c-\e}_p$. Notice that $(f^{-1}(c-\e) \sm \Lambda^{c-\e}_p,\ker\lambda)$ has a natural contact inclusion into $(f^{-1}(c+\e),\ker\lambda)$, defined by the flow of the gradient-like vector field $V_\lambda$. We refer to the monograph \cite{CE} for proofs of these statements and a more complete discussion of Weinstein handle attachments.

In the particular case in which $c\in[0,1]$ is a critical value of $f$ with a unique critical point $p$ of index $n$ then $\Lambda^{c-\e}_p\sse(f^{-1}(c-\e),\ker\lambda)$ is a Legendrian submanifold. If this Legendrian is loose, we say that the critical point $p$ is a \emph{flexible} critical point.

\begin{definition}\label{def:flex}
A Weinstein cobordism $(W,\lambda,f)$ is said to be \emph{flexible} if every critical point of $f$ is either subcritical or flexible.
\end{definition}

\begin{remark}
In $\dim(W)=4$, a critical point $p$ is called flexible if the Legendrian $\Lambda^{c-\e}_p$ has overtwisted complement. This dimension is however not discussed in this paper.\hfill$\Box$
\end{remark}

By Definition \ref{def:flex}, every subcritical Weinstein cobordism is flexible. The importance of Definition \ref{def:flex} is that flexible Weinstein manifolds are completely classified \cite[Chapter 14]{CE}.

In our case, we use flexible Weinstein cobordisms in relation to overtwisted contact manifolds. The first result we need to prove in this direction, which will be used in Section \ref{sec: cobord} for part of the proof of Theorem \ref{thm:main}, is the following proposition.

\begin{prop}\label{prop: flexible ot}
Let $(W,\lambda,f)$ be a flexible Weinstein cobordism such that $(\dd_-W,\ker\lambda)$ is an overtwisted contact manifold. Then the contact manifold $(\dd_+W,\ker\lambda)$ is overtwisted.
\end{prop}

\begin{proof}
First, split the cobordism $(W, \lambda, f)$ into cobordisms with a single critical point
$$W = f^{-1}([0, c_1]) \cup \ldots \cup f^{-1}([c_s, 1]),\quad\mbox{for }0<c_1<\ldots<c_s<1.$$

The resulting attaching spheres $\Lambda_j \sse (f^{-1}(c_j),\ker\lambda)$, for $1\leq j\leq s$, are either subcritical or loose Legendrians submanifolds, and we will now show by induction that each contact manifold $(f^{-1}(c_j),\ker\lambda)$ is overtwisted. The $j=1$ case follows from the fact that $(\dd_- W,\ker\lambda)$ is overtwisted, and the case $j=s$ case implies the result. The contact manifold $(f^{-1}(c_{j+1}),\ker\lambda)$ is obtained from $(f^{-1}(c_j),\ker\lambda)$ by a single Weinstein surgery along the isotropic sphere $\Lambda_j$, and any smooth isotopy of $\Lambda_j$ can be $C^0$--approximated by a contact isotopy. Indeed, if $\Lambda_j$ is subcritical this follows from the $h$--principle for subcritical isotropic submanifolds \cite{PDR}, and if $\Lambda_j$ is a loose Legendrian this is Theorem \ref{thm: c0 loose}. In particular, we can find a contact isotopy which makes the attaching isotropic sphere $\Lambda_j$ disjoint from any overtwisted disk in $(f^{-1}(c_j),\ker\lambda)$.
\end{proof}

Finally, we define a vertical connected sum operation of Weinstein cobordisms. For that, let $(W_1, \lambda_1, f_1)$ and $(W_2, \lambda_2, f_2) $ be two Weinstein cobordisms with non-empty negative boundary, and choose two points $p_1 \in \dd_-W_1$ and $p_2 \in \dd_-W_2$ which are not in the descending manifold of any critical point. Let $\gamma_1$ and $\gamma_2$ be the image curves of the points $p_1$ and $p_2$ by the flow of the gradient-like vector fields $V_{\lambda_1}$ and $V_{\lambda_2}$, and thus $\gamma_1 \sse W_1$ and $\gamma_2 \sse W_2$ are two curves which intersect transversely every level set of their corresponding ambient cobordisms exactly once.

We define the connected sum cobordism as the smooth cobordism
$$W_1 \ol{\#} W_2 = (W_1 \sm \Op(\gamma_1)) \cup (W_2 \sm \Op(\gamma_2)),$$
where the union glues a collar neighborhood of $\dd \Op(\gamma_1)$ to a collar neighborhood of $\dd \Op(\gamma_2)$ with a map that pulls back the Liouville form $\lambda_2$ to the Liouville form $\lambda_1$ and the Morse function $f_2$ to the Morse function $f_1$. The smooth manifold $W_1 \ol{\#} W_2$ then inherits a Weinstein structure $(W_1 \ol{\#} W_2,\lambda, f)$, the critical set of $f$ being the union of the critical sets of $f_1$ and $f_2$, and every regular level set $(f^{-1}(c),\ker\lambda)$ being contactomorphic to the contact connected sum $(f_1^{-1}(c),\ker\lambda_1) \# (f^{-1}_2(c),\ker\lambda_2)$. The Weinstein manifold $(W_1 \ol{\#} W_2,\lambda,f)$ is the \emph{vertical connected sum} of $(W_1,\lambda_1,f_1)$ and $(W_2,\lambda_2,f_2)$.

\begin{remark}
This operation is used in \cite[Section 5]{MNPS} to construct contactomorphisms using the flexible Weinstein $h$-cobordism theorem \cite{CE}, we use the vertical connected sum in Sections \ref{sec: cobord} and \ref{ssec:cons}.\hfill$\Box$
\end{remark}

The connected sum cobordism and flexible Weinstein structures have a crucial role in the cobordism arguments proving Theorem \ref{thm: ps to OT} and Theorem \ref{thm: weinstein exist}. Note also the Weinstein cobordisms are the natural context in which contact surgeries, either positive or negative, arise. Further discussion on contact $(+1)$--surgeries appears in Section \ref{sec: surgery}, but for now we move forward and complete the preliminaries concerning the objects appearing in Theorem \ref{thm:main}, that is, we discuss the statement of the sixth equivalence (1)$\Longleftrightarrow$(6) in Theorem \ref{thm:main}, concerning open book decompositions.

\subsection{Open book decompositions}\label{ssec: OB intro} Open books compatible with a contact structure have a central role in contact topology \cite{Gi2,Wa}. Theorem \ref{thm:main} states that it is possible to characterize higher--dimensional overtwistedness in terms of compatible open book decompositions. In this subsection we review the basic facts about open book decompositions relevant for the statement of Theorem \ref{thm:main} and its proof.

Let $(W, \lambda)$ be a Liouville domain, i.e.~an exact symplectic manifold with the Liouville vector field $V_\lambda$ outwardly transverse to the smooth boundary $\dd W$, and $\p: W \longrightarrow W$ a compactly supported exact symplectomorphism, such that $\p^*\lambda = \lambda + dh$ for some compactly supported function $h\in C^\infty_c(W)$. The triple $(W, \lambda, \p)$ is an \emph{open book decomposition} \cite{Co2,Gi2}, and the Liouville domain $(W,\la)$ is referred to as its {\it page}.

Every open book decomposition $(W,\lambda,\p)$ canonically defines a contact manifold $(Y, \xi)$, which is constructed as the mapping torus
$$Y = W \x [0,1] / (x, 1) \sim (\p(x), 0) \underset{\dd W \x S^1}{\cup} \dd W \x D^2 $$
$$\xi = \ker\left( (\lambda + K d\theta + \theta dh) \cup (\lambda|_{\dd W} + K r^2 d\theta)\right).$$
for a sufficiently large $K\in\R^+$. We write $(Y, \xi) = \op{OB}(W, \lambda, \p)$ to denote this relationship, and say that $(Y,\xi)$ is \emph{compatible with} or \emph{supported by} the open book $(W, \lambda, \p)$. Notice that the construction readily implies the contactomorphism $\op{OB}(W, \lambda, \p) \cong \op{OB}(W, \lambda, \psi\circ\p\circ\psi^{-1})$ for any symplectomorphism $\psi$.

The remarkable feature of open book decomposition in relation to contact structures is that the converse also holds. This is E.~Giroux's existence theorem \cite{Gi2}:

\begin{thm}[\cite{Gi2}]\label{ob exist}
Every contact manifold $(Y, \xi)$ can be presented as $(Y, \xi) = \op{OB}(W, \lambda, \p)$, and there exists a Morse function $f:W\longrightarrow[0,1]$ such that $(W, \lambda, f)$ is a Weinstein manifold.
\end{thm}

Hence the study of contact manifolds can be approached as the study of Weinstein structures and their compactly supported symplectomorphisms. Let $(W, \lambda)$ be a Liouville manifold, and suppose it contains a parametrized Lagrangian sphere $L \sse (W,\lambda)$. We denote the Dehn twist \cite{Co2,Se97} around the Lagrangian sphere $L$ by $\tau_L \in \op{Symp}_c(W)$, where we have extended the Dehn twist $\tau_L\in\op{Symp}_c(\Op(L))$ to a compactly supported symplectomorphism of the ambient Weinstein manifold $(W,\la)$ by using the identity on the complement $W\setminus\Op(L)$.

Note that a Lagrangian sphere $L$ is an exact Lagrangian the moment $\dim(L)\geq2$, and thus $L\sse(W,\la)$ defines a Legendrian sphere $\Lambda$ in the contact manifold $\op{OB}(W, \lambda, \p)$ obtained by integrating the exact form $\lambda|_L$. We denote the relation between the exact Lagrangian $L$ and its Legendrian lift $\Lambda$ by the equality
$$(Y, \xi, \Lambda) = \op{OB}(W, \lambda, \p, L).$$
This equality is defined to contain two statements. First, the contact manifold $(Y,\xi)$ is adapted to the open book decomposition $\op{OB}(W, \lambda, \p)$, where $(W,\lambda)$ is the Liouville page and $\p\in\Symp^c(W,\la)$ is the symplectic monodromy. Second, the Legendrian $\La\sse(Y,\xi)$ is Legendrian isotopic to the Legendrian lift of an exact Lagrangian $L\sse(W,\lambda)$ embedded in the Liouville page.

\begin{remark}
The equality $(Y, \xi, \Lambda) = \op{OB}(W, \lambda, \p, L)$ is not an existence theorem, i.e. it is not meant to state that for {\it any} Legendrian $\La\sse(Y,\xi)$, there exists an open book $\op{OB}(W, \lambda, \p)$ supporting $(Y,\xi)$ and an exact Lagrangian $L\sse(W,\la)$ whose Legendrian lift is isotopic to $\La$. The equality is only used when the existence of such $L\sse(W,\la)$ is known and the equality is the notation we use to specify that data.\hfill $\Box$
\end{remark}

The conjugation invariance stated above Theorem \ref{ob exist} now reads
$$\op{OB}(W, \lambda, \p, L) = \op{OB}(W, \lambda, \psi \circ \p \circ \psi^{-1}, \psi(L))$$
as it can be readily verified by considering $\Lambda$ as being near the page $\theta = 0$.

The following proposition relates Dehn twists of exact Lagrangian on the page of an open book $(W,\la,\p)$ with contact surgeries on the the associated contact manifold:

\begin{prop}[\cite{Ko}]\label{prop: ob surgery}
Suppose that $(Y, \xi, \Lambda) = \op{OB}(W, \lambda, \p, L)$, then the contact manifold $\op{OB}(W, \lambda, \p \circ \tau_L)$ is obtained from $(Y, \xi)$ by contact surgery along $\Lambda$.
\end{prop}

Note that both the mapping class $[\tau_L] \in \pi_0 \op{Symp}_c(W)$ and the contact surgery along $\Lambda$ depend on a parametrizations $S^{n-1} \cong L$ and $S^{n-1} \cong \Lambda$, which is often non-canonical. The diffeomorphism $\Lambda \cong L$ is however canonically given by projection to the page $(W,\lambda)$.

The remaining ingredient to be discussed in relation to compatible open books is the stabilization procedure. Consider a Lagrangian disk $D \sse (W,\lambda)$ with Legendrian boundary $\dd D \sse (\dd W,\ker\lambda)$ and attach a Weinstein handle to $(W,\lambda)$ along the Legendrian sphere $\dd D$, obtaining a new Weinstein manifold $(W \cup H, \lambda')$. Let us assume that the smooth parametrization of the Legendrian boundary $\dd D \sse (\dd W,\ker\lambda)$ is such that the Lagrangian sphere $S$, whose lower hemisphere is the Lagrangian disk $D$ and whose upper hemisphere is the core of the handle $H$, is a smoothly standard sphere. See \cite[Section 6.3]{GP} and \cite{We91} for further details on Weinstein handle attachments.

With this assumption, the new Weinstein manifold $(W \cup H, \lambda')$ contains a Lagrangian sphere $S$, smoothly standard, whose lower hemisphere is the Lagrangian disk $D$ and whose upper hemisphere is the core of the handle $H$. Then, the new open book decomposition $(W \cup H, \lambda', \p \circ \tau_L)$ is said to be the \emph{positive stabilization} of $(W, \lambda, \p)$ along $D$, and $(W \cup H, \lambda', \p \circ \tau_L^{-1})$ is referred to as the \emph{negative stabilization} of $(W, \lambda, \p)$ along $D$ \cite{Co2,Ko}.

Both the positive and the negative stabilization of an open book decomposition can be described as a contact connected sum. This description is the content of the following theorem.

\begin{thm}[E.~Giroux]\label{thm:NegStabCS}
Let $(Y, \xi)=\op{OB}(W, \lambda, \p)$ be a contact manifold, $D\sse(W,\lambda)$ any Lagrangian disk with Legendrian boundary $\dd D \sse (\dd W,\ker\lambda)$, and consider the contact structure $(S^{2n-1}, \xi_-)=\op{OB}(T^*S^{n-1}, \tau^{-1})$.

Then the positive and negative stabilizations of $(W,\lambda, \p)$ along $D$ are diffeomorphic to $Y$. The positive stabilization is contactomorphic to $(Y, \xi)$, and the negative stabilization is contactomorphic to the contact connected sum $(Y \# S^{2n-1}, \xi \# \xi_-)$.\hfill$\Box$
\end{thm}

This result is due to E.~Giroux, but there is no detailed account on it available on the literature, it is however well--known to experts, and an outline of the proof can be found in the article \cite[Proposition 2.6]{CM}. To the knowledge of the authors, E.~Giroux and his collaborators are currently writing a more detailed source.

\section{Thick neighborhoods of overtwisted submanifolds}\label{sec:OTxD2}
In this section we begin the proof of Theorem \ref{thm:main} with the equivalence (1)$\Longleftrightarrow$(2). In transparent terms, the equivalence states that a contact manifold $(Y,\xi)$ is overtwisted if and only if it contains an overtwisted contact submanifold $(N,\xi_\ot)$ with an infinite contact neighborhood. In fact, as noted in Remark \ref{rmk:2a=2b} above, this is equivalent to the existence of an overtwisted contact submanifold with an arbitrarily large, but finite, contact neighborhood. This latter characterization is the result we prove in the main theorem of this section, Theorem \ref{thm:OTxD2}.

In here, we are measuring the size of a contact neighborhood in terms of the maximal radius that can be achieved in the normal form for the contact structure in a neighborhood of a contact submanifold \cite[Section 2.5.3]{Ge}.

\begin{remark}
Technically, the radius exists as a global coordinate only if the conformal symplectic normal bundle is trivial, however in order to detect overtwistedness it suffices to restrict the symplectic normal bundle to a neighborhood of an overtwisted disk in the contact submanifold, in which case the normal bundle becomes trivial and the distance to the zero section provides a well--defined radius coordinate.\hfill$\Box$
\end{remark}

The equivalence (1)$\Longleftrightarrow$(2) thus becomes a statement about the behaviour of overtwisted contact manifolds after a large enough thickening. This first equivalence in Theorem \ref{thm:main} is the content of the following theorem.
\begin{thm}\label{thm:OTxD2}
Let $(N^{2n-1},\ker\alpha_{ot})$ be an overtwisted contact structure. Then for a sufficiently large radius $R\in\R^+$, the contact manifold $(N \x D^2(R), \ker(\alpha_\ot + \lambda_\std))$ is overtwisted. 
\end{thm}

Theorem \ref{thm:OTxD2} and its proof require some preliminaries, including the definition of the higher--dimensional overtwisted disk \cite[Definition 3.6]{BEM}. This definition is reviewed in Subsection \ref{ssec:otdisks}, and we provide the necessary details in this article such that the reader does not need to read \cite{BEM}.

Theorem \ref{thm:OTxD2} is proven in Subsection \ref{ssec:3to5} for the case $n=2$, where it is proven that a sufficiently large neighborhood of an overtwisted contact 3--fold is an overtwisted contact 5--fold. Then we proceed with the general case of Theorem \ref{thm:OTxD2} in Subsection \ref{ssec: highD domain}; this distinction between the 5--dimensional case and higher--dimensions is not essential and we could have written a unified proof for any $n\geq2$. However, encouraged by the suggestions of readers and referees it seems that this distinction contributes to a better understanding of the result.

\begin{remark}
The radius $R\in\R^+$ that appears in the statement depends on the choice of contact form $\alpha_{ot}$ for the contact manifold $(N,\ker\alpha_\ot)$. This dependence is to be expected since there is no natural distance measurement associated to a hyperplane distribution and the usual normalization is to fix a contact form.\hfill$\Box$
\end{remark}

Let us now start by describing the contact germ that defines an overtwisted disk in higher--dimensions, which lies at the core of Theorem \ref{thm:OTxD2}.
\subsection{Overtwisted Disks}\label{ssec:otdisks}
In order to define an overtwisted disk in an arbitrary dimension \cite[Section 3]{BEM} we first consider cylindrical coordinates
$$(z,u_1,\ldots,u_{n-2},\p_1,\ldots,\p_{n-2})\in\R^{2n-3}=\R\x(\R^2)^{n-2}$$
with each pair $(\sqrt{u_i},\p_i)\in\R^2$ being polar coordinates, and note that the standard contact structure $(\R^{2n-3},\xi_\std)$ is given by the kernel of the 1--form
$$\alpha_\std = dz + \sum_{i=1}^{n-2}u_id\p_i = dz + ud\p,\mbox{ where } u :=\displaystyle\sum_{i=1}^{n-2}u_i,\mbox{ and }ud\p:=\sum_{i=1}^{n-2}u_id\p_i.$$
The aim is to define a germ of a contact structure along a $(2n-2)$--dimensional disk, the overtwisted disk, in a $(2n-1)$--dimensional contact manifold. For that, we let $\varepsilon\in\R^+$ be given, consider the contact subdomains of $(\R^{2n-3},\xi_\std)$ given by
$$\Delta_\cyl = \{z \in [-1, 1-\e], \, \, u \in [0,1]\},\quad \Delta_\e = \{z \in [-1+\e, 1-\e], \,\, u \in [0, 1-\e]\}\sse\Delta_\cyl,$$
and define the subset $B = \{z=-1, \, \, u \in [0,1]\} \cup \{z \in [-1, 1-\e], \, \, u = 1\} \sse \dd\Delta_\cyl$ of the boundary of $\Delta_\cyl$. These three contact domains $\Delta_\cyl,\Delta_\e$ and $B$ are shown in Figure \ref{fig:FigureOT1}.
\begin{figure}[h!]
\includegraphics[scale=0.5]{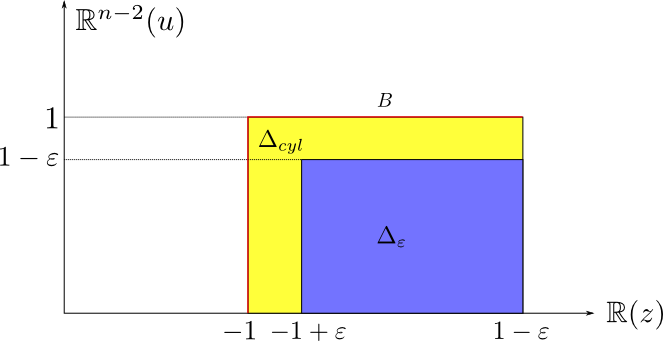} \\
\caption{{\small The domains $\Delta_{cyl}$ in yellow, $\Delta_\varepsilon$ in blue and $B$ in red. The domains are rotationally symmetric along the $z$--axis and we implicitly consider the coordinates of the angle $\p$ as included in the graphic representations of these domains.}} \label{fig:FigureOT1}
\end{figure}

In a nutshell, the contact germ will be defined as the restriction of an ambient contact structure in a neighborhood of a hypersurface describe as the graph of a particular function in the domain $\Delta_\cyl$. Let $k_\e: \R \longrightarrow \R$ be the piecewise linear function defined by
\begin{align*}
	k_\e(x) &:= 
	\begin{cases}
	0 & x \leqslant 1-\e\\
	x-(1-\e) & x \geqslant 1-\e.
	\end{cases}
\end{align*}
and fix a piecewise smooth function $K_\e: \Delta_\cyl \longrightarrow \R$ of the form

\begin{align*}
	K_\e(z,u_1, \p_1,\ldots,u_{n-2},\p_{n-2}) &:= 
	\begin{cases}
	 k_\e(|z|) + k_\e(u) & (z,u_1, \p_1,\ldots,u_{n-2},\p_{n-2}) \in \Delta_\cyl \sm \Int(\Delta_\e)\\
	<0& (z,u_1, \p_1,\ldots,u_{n-2},\p_{n-2}) \in \Int(\Delta_\e).
	\end{cases}
\end{align*}

Let us denote $q = (z,u_1, \p_1,\ldots,u_{n-2},\p_{n-2})$, then the function $K_\varepsilon$ defines the following two embeddings of two ($2n-2$)--dimensional hypersurfaces:
$$\Sigma_1 = \{(q, v, t) \in \Delta_\cyl \x T^*S^1 \, : \, t \in S^1, \, \, v =K_\e(q)\} \sse (\Delta_\cyl \x T^*S^1, \ker(\alpha_\st + vdt))$$
$$\Sigma_2 = \{(q, v, t) \in \Delta_\cyl \x \C \, : \, q \in B, \, \, t \in S^1, \, \, v \in [0, K_\e(q)]\} \sse (\Delta_\cyl \x \C, \ker(\alpha_\st + vdt)).$$

In the description of $\Sigma_1$ the pair of coordinates $(v,t)$ represents linear coordinates in $T^*S^1$ whereas in the definition of $\Sigma_2$ the coordinates $(\sqrt v, t)$ represent polar coordinates on the complex plane $\C$.

\begin{remark}\label{rmk:vt1}
The homonymous notation \cite[Section 2]{BEM} for these two distinct pairs of coordinates is genuinely useful once interiorized, and as the notation suggests we then implicitly identify the open subset $\{v>0\} \sse T^*S^1$ with the open subset $\C^* = \{\sqrt{v} > 0\} \sse \C$.\hfill$\Box$
\end{remark}

Notice that the function $K_\e:\Delta_\cyl\longrightarrow\R$ satisfies $K_\e|_B > 0$ on the subset $B\subseteq\dd \Delta_\cyl$ and thus the hypersurface $\Sigma_2$ is well--defined as a subset of $\Delta_\cyl\times\C$. Each of the two hypersurfaces $\Sigma_1$ and $\Sigma_2$ defines a germ of a contact structure $(\Op\Sigma_1,\eta_1)$ and $(\Op\Sigma_2,\eta_2)$ inherited from its respective ambient contact domains $(\Delta_\cyl \x T^*S^1, \ker(\alpha_\st + vdt))$ and $(\Delta_\cyl \x \C, \ker(\alpha_\st + vdt))$.

By using the contact identification of the two respective subsets
$$\{(v,t)\in T^*S^1:v>0\}\sse(T^*S^1,\la_\std),\quad\{(\sqrt{v},t)\in \C:v>0\}\sse(\C,\la_\std)$$
in the two ambient contact domains $(\Delta_\cyl\times T^*S^1,\ker(\a_\std+\la_\std))$ and $(\Delta_\cyl\times\C,\ker(\a_\std+\la_\std))$, the union $\Sigma_1 \cup \Sigma_2$ of the two hypersurfaces is a piecewise smooth disk in a contact domain and thus, by restriction of the contact structure, we obtain a contact germ in (a neighborhood of) this disk. Let us denote the disk endowed with this germ of a contact structure by $(D_{K_\e}, \eta_{K_\e})$.

\begin{remark}
Note that the dependence of the contact germ on the constant $\e\in\R^+$ is geometrically meaningful. Intuitively it describes the amount of rotation that the contact structure is allowed to have in the boundary $\dd\Delta_\cyl$, and this quantity features crucially in the argument for the existence $h$--principle \cite{BEM,El}.\hfill$\Box$
\end{remark}

Let us move to the definition of the overtwisted disk. In the article \cite[Definition 3.6]{BEM}, an overtwisted disk $(D_{K_\univ}, \eta_{K_\univ})$ is defined to be a certain contact germ $\eta_{K_\univ}$ along a piecewise smooth $(2n-2)$--disk $D_{K_\univ}$, where the definition of the function $K_\univ$ is neither constructive nor canonical. However in this article we can take the function to be
$$K_\univ = K_\e:\Delta_\cyl\longrightarrow\R$$
for any sufficiently small $\e < \e_\univ$, where $\e_\univ$ is a fixed constant depending only on dimension. We then have the following definition:

\begin{definition}
An overtwisted disk $(D^\ot_\e, \eta^\ot_\e)$ is any contact germ along a disk of the form $(D_{K_\e}, \eta_{K_\e})$ where the constant $\e\in\R^+$ satisfies $\e < \e_\univ$.\hfill$\Box$
\end{definition}

In practice, this implies that finding an overtwisted disk is tantamount to finding a neighborhood of a disk with the contact germ $(D_{K_\e}, \eta_{K_\e})$ for an arbitrarily small $\e\in\R^+$. Note also that the contact structure is defined as a contact germ on the disk, and thus the smooth regularity of the disk, as a hypersurface, is not a concern from the smooth topology perspective: the contact structure is defined in a smooth open neighborhood of a disk, which is still a smooth neighborhood even if the disk we consider is piecewise smooth \cite{BEM}.

\begin{definition}\label{def:ot}
A contact manifold $(Y, \xi)$ is \emph{overtwisted} if there exists a piecewise smooth embedding $D^{2n-2} \sse Y$ such that the contact germ $(D^{2n-2}, \xi|_D^{2n-2})$ is an overtwisted disk.\hfill$\Box$
\end{definition}

The reader should now be equipped to understand the statement of Theorem \ref{thm:OTxD2} and thus the statement of the equivalence (1)$\Longleftrightarrow$(2) in Theorem \ref{thm:main}.

Let us proceed with the proof of Theorem \ref{thm:OTxD2} in the case that $(N,\xi_\ot)$ is a 3--dimensional overtwisted contact manifold, which corresponds to the characterization (1)$\Longleftrightarrow$(2) in the case where $(Y,\xi)$ is a 5--dimensional contact manifold in the statement of Theorem \ref{thm:main}.

\subsection{The 5-dimensional case}\label{ssec:3to5}
The initial step in order to prove Theorem \ref{thm:OTxD2} for the $n=2$ case is to substitute the general overtwisted contact 3--fold $(N,\xi_\ot)$ by an explicit overtwisted local model $(M,\ker\a_M)$ and prove Theorem \ref{thm:OTxD2} for this particular 3--dimensional overtwisted contact domain $(M,\ker\a_M)$.

The description and motivation of this contact domain $(M,\ker\a_M)$ strongly use the bivalent coordinates $(v,t)$ that appear in Subsection \ref{ssec:otdisks}, which allows us to neatly describe the transition from $(T^*S^1,\la_\std)$ to the complex plane $(\C,\la_\std)$.

In this 5-dimensional case, the overtwisted disk $(D^2_\ot,\xi_\ot)$ is equivalent to the 2-dimensional overtwisted disk introduced by Y.~Eliashberg \cite[Section 1.4]{El} with the singular characteristic foliation as depicted in \cite[Section 4.5]{Ge}. The model for $(D^2_\ot,\xi_\ot)$ that we have in mind in the present article is \cite[Figure 4.9]{Ge}, i.e. an embedded 2-disk $(D^2_\ot,\xi_\ot)$ whose characteristic foliation contains a unique singular point in the interior and the characteristic foliation of $(D^2_\ot,\xi_\ot)$ is singular along $\partial D^2_\ot$. In particular, $\partial D^2_\ot$ is a Legendrian curve with vanishing Thurston-Bennequin invariant.

\begin{remark}\label{rmk:vt2}
In Definition \ref{def:ot} for the overtwisted disks we expressed the contact germ in a particular disk $(D^\ot_\e, \eta_\e^\ot) = \Sigma_1 \cup \Sigma_2$, which is described as the union of two pieces $\Sigma_1$ and $\Sigma_2$. These two pieces are both defined in terms of a function $K_\e:\Delta_\cyl\longrightarrow\R$ as explained above: the first piece $\Sigma_1$ is precisely the graph of the function $\{v = K_\e\} \sse \Delta_\cyl \x T^*S^1$, whereas the second piece $\Sigma_2$ is instead the sublevel set $\{ v \les K_\e|_B\} \sse \Delta_\cyl \x \C$.

For the first piece $\Sigma_1$, it is essential that the coordinates $(v,t)$ belong to the cotangent space $(v,t) \in T^*S^1$, and not the complex plane, since the function $K_\e:\Delta_\cyl\longrightarrow\R$ attains both positive and negative values. In contrast, for the second piece $\Sigma_2$ it is essential that the coordinates $(\sqrt v, t)$ actually define polar coordinates on the complex plane $(\C,\la_\std)$, since we can then define the sublevel set $\{v \les K_\e\}$ correctly. This hopefully emphasizes the importance of the varying domains that the coordinates $(v,t)$ are defining.\hfill$\Box$
\end{remark}

Let us now describe a local model $(M, \ker\alpha_M)$ which is contained in any overtwisted contact $3$-fold $(N,\xi_\ot)$ and has a crucial role in the proof of Theorem \ref{thm:OTxD2}. The domain $M$ is diffeomorphic to a compact $3$-ball with a piecewise smooth boundary and it admits coordinates $(z,v,t)$, where $(v,t)$ are coordinates in the sense of Remarks \ref{rmk:vt1} and \ref{rmk:vt2} above. In these coordinates $(z,v,t)$ the contact form $\a_M$ reads
$$\alpha_M = dz + vdt.$$

It is our duty to be precise with the meaning of the pair $(v,t)$: in this case the coordinate $z \in (-3-\e, 1]$ dictates the domain of definition of the pair of coordinates $(v,t)$. This goes as follows, the symplectic submanifolds $\{z = \const\}$ belong to one of these three types:

\begin{figure}[h!]
\includegraphics[scale=0.5]{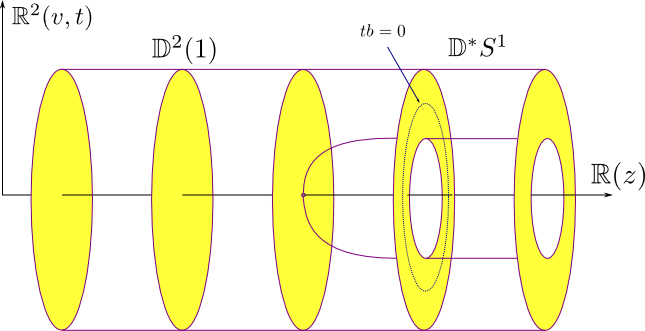}
\caption{{\small The overtwisted contact ball $(M,\ker\alpha_M)$}.}\label{fig:FigureOT0}
\end{figure}

\begin{itemize}
\item[a.] For $z \in [-1 + \frac{2\e}3, 1]$, we have $(v,t)\in[-1,1]\times S^1$. Thus in this range the submanifolds $\{z = \const\}$ are exact symplectomorphic to the unit disk bundle $(D^*S^1,\la_\std)$ inside $(T^*S^1,\la_\std)$ since the restriction of $\alpha$ equals the canonical Liouville form.\\

\item[b.] For $z \in (-1 + \frac\e3, -1 + \frac{2\e}3)$, we let $t \in S^1$ and $v \in [\e, 1]$. Then the fibers are exact symplectomorphic to $\{\e \leq v \leq 1\} \sse (T^*S^1,\la_\std)$. Notice that these fibers are also equal to the standard Liouville structure on $D^2(1) \sm D^2(\e^2) \sse \C$, where we equip $(\C,\la_\std)$ with polar coordinates $(\sqrt v, t)$.\\

\item[c.] For $z \in [-3 - \e, -1 + \frac\e3)$, we define the fibers $\{z = \const\}$ to be equal to the unit disk $(D^2(1),\la_\std)\sse(\C,\la_\std)$, with $(\sqrt v, t)$ continuing to represent polar coordinates.\\
\end{itemize}

The choice of the dependence of the domain of $(v,t)$ on the $z$-coordinate allows for a more flexible notation, which hopefully helps the reader. It is possible to alternatively define $(v,t)$ to be global coordinates independent of $z$ and work with a contact form
$$\alpha_M = dz + \rho(v)dt,$$
and domain $M$ depending on the choice of a Hamiltonian $\rho:\R\longrightarrow\R$. This is the notation that is followed in \cite[Section 2]{BEM}. In our notation, the dependence of $(v,t)$ absorbs keeping track of such Hamiltonian $\rho$, which we prefer.

\begin{remark}
It might help the reader to understand the contact domain $(M,\a_M)$ as a symplectic foliation where the leaves are parametrized by the interval in the $z$--coordinate. This fibration viewpoint has been fruitful in contact topology \cite{CaPhD,CPP} and it provided us with the right insight to prove Theorem \ref{thm:OTxD2}.\hfill$\Box$
\end{remark}

The contact domain $(M,\ker\alpha_M)$ is depicted in Figure \ref{fig:FigureOT0}, where the reader can see how the dependence of the domain of the coordinates $(v,t)$ varies according to the value of the $z$--coordinate. Let us now analyze the two fundamental contact properties of $(M,\ker\a_M)$:

\begin{itemize}
\item[1.] First, the 3--dimensional contact domain $(M,\ker\alpha_M)$ is overtwisted.\\
 
\noindent This can be proven by direct inspection and finding an overtwisted 2--disk. Instead, we can note that the Legendrian circle $\{(z,v,t)\in M: 6z = -6 + 5\e, v = 0\}\sse (M,\a_M)$ is an unknotted Legendrian with zero Thurston-Bennequin number, which proves that the contact model $(M,\a_M)$ is overtwisted \cite{El}.\\
 
\item[2.] Second, the contact domain $(M,\ker\alpha_M)$ serves as a local model in any overtwisted $3$-manifold. Indeed, if $(Y, \ker\alpha)$ is any overtwisted contact $3$-manifold, possibly open, then $(M,\ker\a_M)$ admits a contact embedding
 $$f:(M,\ker\a_M)\longrightarrow(Y,\ker\a)$$
 due to Eliashberg's classification theorem \cite{El}. Even better, defining the positive smooth function $c_f:M\longrightarrow\R$ given by the conformal factor $f^*\alpha = c_f\alpha_M$, we conclude by compactness that there exists a constant $R\in\R^+$ such that $c_f < R$. In consequence, the contact product $(M \x D^2(1),\ker(\a_M+\la_\std))$ embeds into the contact product $(Y \x D^2(R),\ker(\a+\la_\std))$. 
\end{itemize}

It follows from the second property and Definition \ref{def:ot} that in order to prove Theorem \ref{thm:OTxD2} in this 5--dimensional case it suffices to find an overtwisted disk in the contact product manifold
$$(M \x D^2(1),\ker(\a_M+\la_\std)).$$
This is our goal now, which we achieve by first proving the technical Lemma \ref{lem: 3D contact}.

Let us define the 3--dimensional contact domain $\wt \Delta = (-3 - \e, 1) \x D^2(1)$ endowed with coordinates $(z,u_0,\p_0)$ and the standard contact form $dz + u_0 d\p_0$, where $z \in (-3 - \e, 1)$ and $(\sqrt u_0, \p_0)$ are polar coordinates on $D^2(1)$. Consider the map
$$\pi:(M\times D^2(1),\a_M+u_0 d\p_0)\longrightarrow (\wt\Delta,dz+u_0 d\p_0),$$
$$(z,v,t,u_0,\p_0)\longmapsto \pi(z,v,t,u_0,\p_0)=(z,u_0,\p_0),$$
whose fibers are Liouville surfaces symplectomorphic to subdomains of $(T^*S^1,\la_\std)$ or $(\C,\la_\std)$.

The contact domain $\wt\Delta$ contains two different subdomains in terms of the fibers of the smooth map $\pi$, which we can define as
$$\wt \Delta_1 = \{(z,u_0,\p_0)\in \wt\Delta:z \in (-1 + 2\e/3, 1-\e]\} \sse \wt \Delta$$
$$\wt \Delta_2 = \{(z,u_0,\p_0)\in \wt\Delta:z \in [-3, -1 + \e/3)\} \sse \wt\Delta.$$
Therefore the first subdomain $\wt \Delta_1 \sse \wt \Delta$ corresponds to those values of $(z, u_0, \p_0)$ such that the fiber is given by $(v,t) \in (T^*S^1,\la_\std)$ for $(z, v, t, u_0, \p_0) \in M \x D^2(1)$. Similarly, the second subdomain $\wt\Delta_2\sse\wt\Delta$ corresponds to those values in $\wt\Delta$ where the fiber of the projection $\pi$ is equivalent to $(\C,\la_\std)$.

\begin{remark}
Following Definition \ref{def:ot} and the discussion above, the proof of Theorem \ref{thm:OTxD2} in this 5--dimensional case consists in finding a 4--dimensional overtwisted disk in the contact model $(M\times D^2(1),\ker(\a_M+\la_\std))$. The contact germ of an overtwisted disk is given in terms of a domain of definition $\Delta_\cyl$, and observe that there is a natural embedding $\Delta_\cyl\longrightarrow\wt\Delta$.

However, and that is the difficulty that needs to be solved at this point, it is not true that we have the inclusions $\Delta_\e\sse\wt\Delta_1$ and $B\sse\wt\Delta_2$. Note that if this were the case the contact model above would readily be overtwisted.\hfill$\Box$
\end{remark}

The exact relation between the contact domain $\wt\Delta$ and $\Delta_\cyl$ needed in order to prove Theorem \ref{thm:OTxD2} in this $n=2$ case is established in the following lemma.

\begin{lemma} \label{lem: 3D contact}
There exists a strict contact embedding $f:(\Delta_\cyl,\xi_\std) \longrightarrow (\wt \Delta,\xi_\std)$, i.e.~ such that $f^*\alpha_\std = \alpha_\std$, with the property that $f(\Delta_\e) \sse \wt \Delta_1$ and $f(B) \sse \wt\Delta_2$.
\end{lemma}

Lemma \ref{lem: 3D contact} will be proven momentarily, but let us first conclude Theorem \ref{thm:OTxD2} for an overtwisted contact 3--fold $(N,\ker\a_\ot)$ assuming such contact embedding exists.

\begin{proof}[Proof of Theorem \ref{thm:OTxD2} for $\dim(N)=3$] It suffices to show that the 5--dimensional domain
$$(M\x D^2(1),\ker(\a_M+u_0d\p_0))$$
is overtwisted, as we have discussed above. In order to do that, let $\varepsilon\in\R^+$ be a fixed but small enough constant such that $\e<\e_\univ$ and consider the contact embedding $f:\Delta_\cyl\longrightarrow \wt\Delta$ provided in Lemma \ref{lem: 3D contact}. The claim is now that the preimage
$$\pi^{-1}(\Op(f(\Delta_\cyl))\sse(M\x D^2(1),\ker(\a_M+u_0d\p_0))$$
contains an overtwisted disk. Indeed, define the function $\wt K_\e: f(\Delta_\cyl) \longrightarrow \R$ as the pull--back
$$\wt K_\e = K_\e \circ f^{-1}:f(\Delta_\cyl)\longrightarrow\R.$$
and consider the two hypersurfaces
$$\wt \Sigma_1 = \{(z, v, t, u_0, \p_0) \, : \, (z, u_0, \p_0) \in f(\Delta_\cyl), \, \, t \in S^1, \, \, v = \wt K_\e(z, u_0, \p_0)\}\sse f(\Delta_\cyl)\times T^*S^1,$$
$$\wt \Sigma_2 = \{(z, v, t, u_0, \p_0) \, : \, (z, u_0, \p_0) \in f(B), \, \, t \in S^1, \, \, v \in [0, \wt K_\e(z, u_0, \p_0)]\}\sse f(B)\times \C.$$

Notice that the first hypersurface $\wt \Sigma_1$ is a well--defined subset of $M \x D^2(1)$ precisely because $f(\Delta_\e) \sse \wt \Delta_1$, and similarly the second hypersurface $\wt \Sigma_2$ is well--defined because the inclusion $f(B) \sse \wt \Delta_2$ is satisfied. This construction now exhibits an overtwisted disk in our 5--dimensional domain $(M\x D^2(1),\ker(\a_M+u_0d\p_0))$: the contact germ of the 4--disk $D^4 = \wt \Sigma_1 \cup \wt \Sigma_2$ obtained as the union of the two hypersurfaces is an overtwisted disk. Indeed, since the 3--dimensional contactomorphism $f:\Delta_\cyl\longrightarrow\wt\Delta$ preserves the contact form, the extended contactomorphism in 5--dimensions
$$F:(\Op(D^\ot_\e),\eta^\ot_\e)\cong (\Op(\Sigma_1\cup\Sigma_2),\eta^\ot_\e)\longrightarrow (M\times D^2(1),\a_M+u_0d\p_0)$$
$$(z, u_0, \p_0, v, t)\longmapsto F(z, u_0, \p_0, v, t) = (f(z,u_0,\p_0), v, t).$$
maps the contact germ $(D^\ot_\e,\eta^\ot_\e)$ to the contact germ $(D^4,\ker(\alpha_M+\lambda_\std))$, as required.
\end{proof}

This concludes the proof of Theorem \ref{thm:OTxD2} in the case that $\dim(N)=3$ modulo the construction of the contactomorphism in Lemma \ref{lem: 3D contact}, which we now prove.

\begin{figure}[h!]
\includegraphics[scale=0.4]{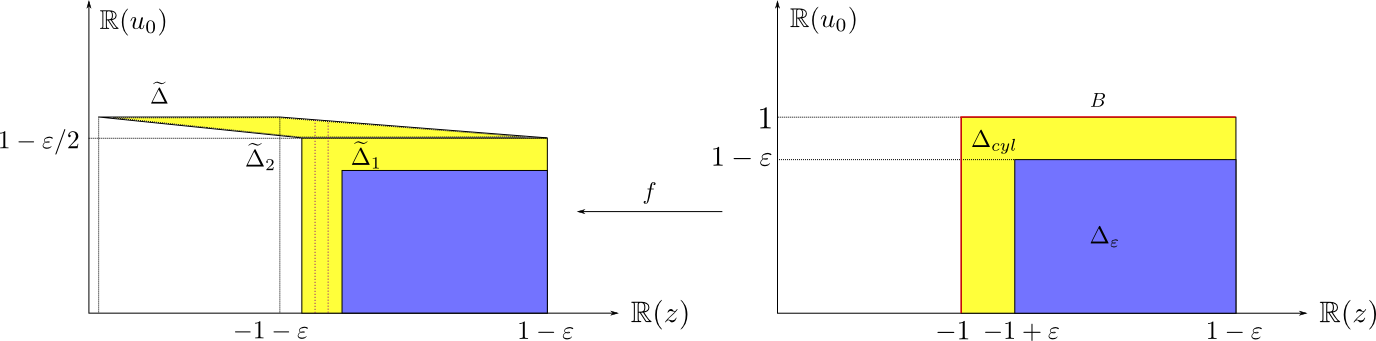}
\caption{{\small The contact domains appearing in Lemma \ref{lem: 3D contact}}.}
\label{fig:FigureOT2}
\end{figure}

\begin{proof}[Proof of Lemma \ref{lem: 3D contact}:]
Let $g:[0,1]\longrightarrow[0,2]$ be a smooth and increasing function which is $C^0$--close to the piecewise linear function defined by $$u \longmapsto \begin{cases} 0 & \text{ if } u \in [0, 1- \frac\e 2] \\ \frac4\e(u-1+\frac\e2) & \text{ if } u \in [1 - \frac\e2, 1]\end{cases},$$
and consider the diffeomorphism $f:\Delta_\cyl\longrightarrow\wt\Delta$ defined by
$$ f(z, u_0, \p_0) = \left(z - g(u_0), u_0, \p_0 - \int_{u_0}^1 \frac{g'(u_0)}{u_0} du_0 \right).$$
The map $f:\Delta_\cyl\longrightarrow\wt\Delta$, which is depicted in Figure \ref{fig:FigureOT2}, has the desired properties from the statement. Indeed, the diffeomorphism $f$ is a $C^0$--approximation of a piecewise smooth contactomorphism which acts by taking the region $\{u \geq 1-\frac\e 2\}$ and shearing its $z$--coordinate far to the left, thus conforming to the required properties.
\end{proof}

\subsection{General dimensions}\label{ssec: highD domain}
In this section we prove Theorem \ref{thm:OTxD2}. The reader is strongly encouraged to have understood the case $n=2$, proven in the previous Subsection \ref{ssec:3to5}. The argument we use in order to conclude Theorem \ref{thm:OTxD2} for an arbitrary overtwisted contact manifold $(N,\xi_\ot)$ contains the same steps as the 5--dimensional case above, but the general higher--dimensional versions of the boundary piece $B\sse\Delta_\cyl$ and Lemma \ref{lem: 3D contact} contain more information.

The first difference between the general and 5--dimensional cases is that the contact embedding $f:\Delta_\cyl\longrightarrow\wt\Delta$ that we use in the general case, generalizing Lemma \ref{lem: 3D contact}, is no longer strict, and thus a conformal factor must be accounted for when constructing the domains to which we push--forward the function $K_\e:\Delta_\cyl\longrightarrow\R$. This conformal factor is the reason for the appearance of the constant $\rho\in\R^+$ in the following definition of the local model $(M,\a_M)$.

Consider two positive reals constants $\e,\rho\in\R^+$, where $\e$ is to be small and $\rho$ quite large. Define the $(2n-3)$--dimensional domain
$$I=(-\rho,1)\times D^{2n-4}(\rho)$$
with coordinates $z\in(-\rho,1)$ and $(u,\p)=(u_1,\p_1,\ldots,u_{n-2},\p_{n-2})$ are polar coordinates on the ball $D^{2n-4}(\rho)$. The domain $I$ generalizes the $z$--coordinate interval in the proof of the 5--dimensional case. Following the first step of the proof in Subsection \ref{ssec:3to5}, we describe a $(2n-1)$--dimensional contact model domain $(M, \alpha_M)$, which is endowed with coordinates $(z, u, \p, v, t)$ and the contact form given globally by
$$\alpha_M = dz + u d\p+ v dt.$$
Hopefully, the reader noticed that the domain of the coordinates $(v,t)$ must at least depend on the $z$--coordinate, as in Subsection \ref{ssec:3to5}. Indeed, the variables $(z,u,\p)\in I$ belong to the fixed domain $I$ but the domain of the variables $(v,t)$ will either be the unit disk bundle $D^*S^1$, the positive part $D^2(1) \sm D^2(\e)$ or the unit disk $D^2(1)$ depending on the coordinates $(z,u,\p) \in I$. This precise dependence is given as follows:

\begin{itemize}
\item[a.] For $(z,|u|)\in I_1:=[-1+\frac{2\varepsilon}{3},1]\x[0,1-\frac{2\e}{3}]\sse I$, we have $(v,t)\in[-1, 1]\times S^1$. In this range, the symplectic submanifolds $\{(z,u,\p)=\mbox{constant}\}$ are exact symplectomorphic to $(D^*S^1,\la_\std)$.\\

\item[b.] For $(z,|u|)\in I_{1/2}:=(-1+\frac{\varepsilon}{3},-1+\frac{2\varepsilon}{3})\x[0,1-\frac{\e}{3})\cup(-1+\frac{\varepsilon}{3},1)\x(1-\frac{2\varepsilon}{3},1-\frac{\e}{3})\sse I$, we consider $(v,t)\in[\delta,1]\times S^1$. The symplectic submanifolds $\{(z,u,\p)=\mbox{constant}\}$ are symplectomorphic to $(D^2(1) \sm D^2(\delta),\la_\std)$, where $\delta\in\R^+$ is a small constant which will be chosen in the proof of Theorem \ref{thm:OTxD2}. The constant $\delta$ does not have a crucial role, and thus we do not include it in the notation.\\

\item[c.] For $\{(z,|u|)\in I_2:=[-\rho,-1+\frac{\varepsilon}{3}]\x[0,\rho]\cup[-\rho,1]\times[1-\frac{\e}{3},\rho]\}\sse I$, we declare $(v, t) \in (D^2(1),\la_\std)$ to be polar coordinates $(\sqrt v, t)$ in the unit disk.\\
\end{itemize}

This $(2n-1)$--dimensional contact local model $(M,\ker\alpha_M)$ has the two properties of its $3$--dimensional analogue in Subsection \ref{ssec:3to5}. First, the contact manifold $(M,\ker\alpha_M)$ is overtwisted if we choose $\rho\in\R^+$ large enough. Second, for any choice of positive constants $\rho,\delta$ and $\e\in\R^+$, this contact local model $(M,\ker\alpha_M)$ exists in every overtwisted $(2n-1)$--dimensional contact manifold $(N,\ker\alpha_\ot)$ by the isocontact embedding $h$--principle \cite[Corollary 1.4]{BEM}, and the fact that the scaling factor between $\alpha_\ot$ and $\alpha_M$ is bounded because $M$ is compact. Hence in order to conclude Theorem \ref{thm:OTxD2} it remains to prove that the $(2n+1)$--dimensional contact domain
$$(M\times D^2(R),\ker(\alpha_M+\lambda_\std))$$
contains an overtwisted $2n$--disk when $R\in\R^+$ is sufficiently large.

Let us introduce the domain $\wt\Delta$ and its relatives, following the steps in Subsection \ref{ssec:3to5}. We consider the contact domain
$$\wt \Delta = (-\rho, 1-\e) \x D^{2n-4}(\rho) \x D^2(R),\mbox{ with coordinates }(z,u,\p,u_0,\p_0),$$
which contains the two $(2n-1)$--dimensional subdomains
$$\wt\Delta_1 = I_1 \x D^2(R) \sse \wt \Delta,\quad \wt\Delta_2 = I_2 \x D^2(R) \sse \wt \Delta.$$
These two subdomains have the same role as their homonymous domains have in the 5--dimensional case discussed in Subsection \ref{ssec:3to5}. Indeed, the reason for considering the two subdomains $\wt\Delta_1$ and $\wt\Delta_2$ is that the symplectic type of the fibers of the projection map
$$\pi:M\times D^2(1)\longrightarrow\wt\Delta,\quad (z,u,\p,v,t,u_0,\p_0)\longmapsto\pi(z,u,\p,v,t,u_0,\p_0)=(z,u,\p,u_0,\p_0)$$
depends on the point of the domain $\wt\Delta$. Indeed, over the region $\wt\Delta_1$ the fiber of the map $\pi$ is the unit cotangent bundle $(D^*S^1,\la_\std)$, whereas over the region $\wt\Delta_2$ the fiber is the unit disk $(D^2(1),\la_\std)$.

In the same vein than Lemma \ref{lem: 3D contact}, we now need to compare the two $(2n-1)$--dimensional contact domains $(\wt \Delta,\xi_\std)$ and $(\Delta^{2n-1}_\cyl,\xi_\std)$, and contact embed the domain $\Delta_\cyl$ inside $\wt\Delta$ in such a manner that the images of the subdomain $\Delta_\e$ and the boundary piece $B$ lie in the appropriate regions of the target domain $\wt\Delta$. This is the content of the following lemma, which generalizes Lemma \ref{lem: 3D contact}.

\begin{figure}[h!]
\includegraphics[scale=0.5]{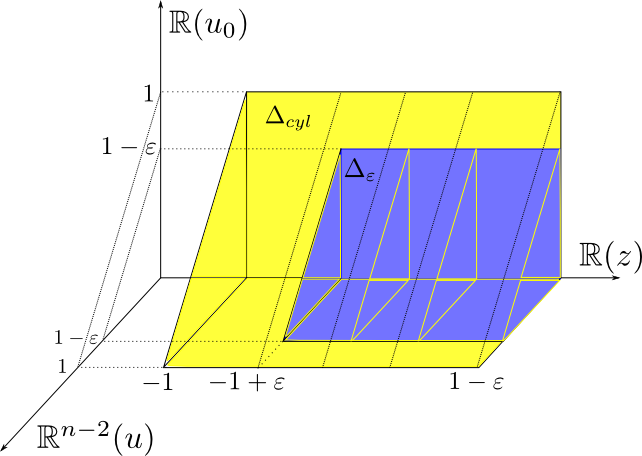} \\
\caption{{\small The contact domain $\Delta_\cyl=(z,u_0,\p_0,u,\p)$.}} \label{fig:FigureOT3}
\end{figure}

\begin{lemma}\label{lem: highD emb}
For any constant $\e\in\R^+$, there exist constants $\rho,R\in\R^+$ such that there is a contact embedding $f:(\Delta_\cyl,\xi_\std)\longrightarrow(\wt\Delta,\xi_\std)$, satisfying $f(\Delta_\e) \sse \wt \Delta_1$ and $f(B) \sse \wt \Delta_2$.
\end{lemma}

\begin{proof}
First, the two $(2n-1)$--dimensional domains $\Delta^{2n-1}_\cyl$ and $\wt \Delta$ are contact subdomains of the ambient contact space $(\R^{2n-1},\xi_\std)\cong(\R^{2n-3}\times D^2(R),\xi_\std)$ and we are using coordinates $(z, u, \p,u_0, \p_0)$, where $(\sqrt{u_0}, \p_0)$ are polar coordinates of the $D^2(R)$ factor.

In comparison to Lemma \ref{lem: 3D contact}, additional effort must be invested when working with the set $B \sse \Delta_\cyl$, which can be described by the union $B = \{z = -1\} \cup \{u_0 + |u| = 1\}$. The reader is encouraged to visualize the subset $B$ in Figure \ref{fig:FigureOT3}, where we have depicted the coordinates $(z,u,u_0)$.

In order to achieve the condition $f(B)\sse\wt\Delta_2$ we have the choice of either decreasing the $z$--coordinate below the value $-1+\e/3$ or increasing the $u$--coordinate beyond the value $1-\e/3$. In fact, we shall use both depending on the region of the set $B\sse\Delta_\cyl$ we find ourselves in. The decomposition of the set $B=B_1\cup B_2$ we consider is defined as follows
$$B_1:=\{(z,u,\p,u_0,\p_0)\in B:|u|\leq\e^2/4\},\quad B_2:=\{(z,u,\p,u_0,\p_0)\in B:|u|\geq\e^2/4\}\sse B.$$
The required contactomorphism $f:(\Delta^{2n-1}_\cyl,\xi_\std)\longrightarrow(\wt\Delta,\xi_\std)$ will be obtained as the composition of two contactomorphisms $g_1:\Delta_\cyl\longrightarrow\wt\Delta$ and $g_2:(\R^{2n-1},\xi_\st)\longrightarrow(\R^{2n-1},\xi_\st)$, both of which will restrict to the identity $g_1|_{\Delta_\e}=g_2|_{g_1(\Delta_\e)}=\mbox{id}|_{\Delta_\e}$ in the region $\Delta_\e=g_1(\Delta_\e)\sse\Delta_\cyl$.
In geometric terms, the contactomorphism $g_1$ will decrease the $z$--coordinate in the region $B_1$ in order to contact embed it into $\wt\Delta_2$, and the contactomorphism $g_2$ will increase the modulus $u$ and embed the region $B_2$ into $\wt\Delta_2$. Let us start with $g_1$, which already featured in the 5--dimensional case.

Consider the contactomorphism $h:(\Delta^3_\cyl,\xi_\std)\longrightarrow(\wt\Delta^3,\xi_\std)$ constructed in Lemma \ref{lem: 3D contact} and define the contactomorphism 
$$g_1:(\Delta^{2n-1}_\cyl,\xi_\std)\longrightarrow(\wt\Delta^{2n-1},\xi_\std),\quad g_1(z, u_0, \p_0, u, \p) = (h(z, u_0, \p_0), u, \p).$$
This contactomorphism satisfies $g_1(B_1)\sse\wt\Delta_2$ since a point $(z, u_0, \p_0, u, \p)\in B_1$ must have $u_0=1-|u|\geq 1-\e^2/4$ and thus the image point $(h(z, u_0, \p_0), u, \p)$ has a $z$--coordinate below the value $-1+\e/3$.

Let us now describe the contactomorphism $g_2$, where we will push the remaining piece $B_2\sse B$ into the region $\wt\Delta_2$. Consider the contact vector field $X=u \partial_u+u_0\partial_{u_0}+z\partial_z$ on $\Delta^{2n-1}_\cyl$ and cut--off its contact Hamiltonian $H\in C^\infty(\Delta^{2n-1}_\cyl)$ to a Hamiltonian $\wt H\in C^\infty(\Delta^{2n-1}_\cyl)$ such that its associated contact vector field $\wt X$ satisfies
\begin{itemize}
\item[a.] $\wt X$ vanishes in the region $\{z\geq-1+\frac{2\varepsilon}3,u+u_0\leq 1-\frac{2\varepsilon}3\}$.
\item[b.] $\wt X$ coincides with the radial vector field $X$ in the region $\{z\leq-1+\frac\varepsilon3,1-\frac\varepsilon3\leq u+u_0\}$.
\end{itemize}

\begin{figure}[h!]
\includegraphics[scale=0.6]{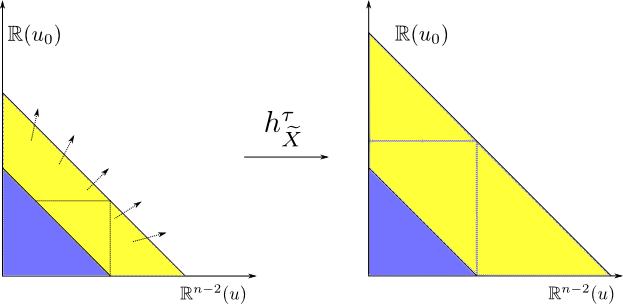} \\
\caption{{\small Cross section $(u,u_0)$ for the expanded domain $h_{\wt X}^\tau(\Delta_\cyl)$.}} \label{fig:FigureOT4}
\end{figure}

Denote by $h^\tau_{\wt X}:(\R^{2n-1},\xi_\st)\longrightarrow(\R^{2n-1},\xi_\st)$ be the $\tau$--time contact flow of the contact vector field $\wt X$: near the region $B$ the contact flow $h^\tau_{\wt X}$ acts as radial expansion, as depicted in Figure \ref{fig:FigureOT4}. The contactomorphism $g_2$ is defined to be $h^{\tau_0}_{\wt X}$ for a large enough time $\tau_0\in\R^+$, and we claim that for such $g_2:(\R^{2n-1},\xi_\st)\longrightarrow(\R^{2n-1},\xi_\st)$ the composition
$$f = g_2\circ g_1:(\Delta_\cyl,\xi_\std)\longrightarrow(\wt\Delta,\xi_\std)$$
satisfies the properties in the statement of the lemma. First, we do have the inclusion $f(\Delta_\e)=g_2(g_1(\Delta_\e))=g_2(\Delta_\e)= \Delta_\e \sse \wt\Delta_1$ since both $g_1$ and $g_2$ are the identity in $\Delta_\e$ by construction. Second, we need to verify that the inclusion $f(B_2)=g_2(g_1(B_2)) \sse \wt\Delta_2$ is satisfied. Indeed, since the $u$--coordinate on the set $g_1(B_2)$ is bounded below by a positive number and the contact flow $g_2$ expands the coordinate $u$ exponentially by construction, we conclude that for large $\tau\in\R^+$ the inclusion $g_2(g_1(B_2)) \sse \{u > 1-\e/3\} \sse \wt \Delta_2$ holds.
\end{proof}

\begin{remark}
The contact embedding $f:(\Delta_\cyl,\xi_\std)\longrightarrow(\wt\Delta,\xi_\std)$ in Lemma \ref{lem: highD emb} is not a strict, in contrast to the contact embedding in Lemma \ref{lem: 3D contact}.\hfill$\Box$
\end{remark}

Lemma \ref{lem: highD emb} is the technical ingredient in order to prove Theorem \ref{thm:OTxD2}, which we now do. The structure of the proof is the same as for its 5--dimensional analogue proven in Subsection \ref{ssec:3to5}. Let us now provide the details.

\begin{proof}[Proof of Theorem \ref{thm:OTxD2}:]
First, we choose a constant $\e\in\R^+$ such that $\e < \e_\univ$ and consider the contact embedding $f:(\Delta^{2n-1}_\cyl,\xi_\std)\longrightarrow(\wt\Delta^{2n-1},\xi_\std)$ constructed in Lemma \ref{lem: highD emb}. Denote by $\rho,R\in\R^+$ the homonymous constants appearing in its statement and consider the conformal factor $c_f \in C^\infty(\Delta_\cyl)$ defined by $f^*\alpha_\std = c_f\alpha_\std$. Now we can proceed as in the 5--dimensional case by defining the Hamiltonian
$$\wt K:f(\Delta^{2n-1}_\cyl)\longrightarrow\R,\quad \wt K = (c_f \cdot K)\circ f^{-1}.$$

The statement of the theorem will be proven if we can find a $2n$--dimensional overtwisted disk in the $(2n+1)$--dimensional contact domain $((M,\a_M)\times D^2(R),\ker(\a_M+\la_\std))$. In order to exhibit the disk, we consider the two domains

$$\wt \Sigma_1 = \{(z, v, t, u, \p) \, : \, (z, u, \p) \in f(\Delta_\cyl), \, \, t \in S^1, \, \, v = \wt K(z, u, \p)\}$$
$$\wt \Sigma_2 = \{(z, v, t, u, \p) \, : \, (z, u, \p) \in f(B), \, \, t \in S^1, \, \, v \in [0, \wt K(z, u, \p)]\}.$$

Notice that for a sufficiently small $\delta\in\R^+$, which appears in the definition of the contact domain $(M,\a_M)$, the hypersurface $\wt\Sigma_1$ is a well--defined subset of $M \x D^2(1)$ since $f(\Delta_\e) \sse \wt \Delta_1$. The second hypersurface $\wt \Sigma_2$ is also well--defined since we have the inclusion $f(B) \sse \wt \Delta_2$. Now the union $D^{2n}=\wt\Sigma_1\cup\wt\Sigma_2$ of the two hypersurfaces $\wt\Sigma_1$ and $\wt\Sigma_2$ endowed with the ambient contact structure is contactomorphic to an overtwisted disk since the map
$$F:(\Op(D^\ot_\e),\eta^\ot_\e)\cong (\Op(\Sigma_1\cup\Sigma_2),\eta^\ot_\e)\longrightarrow (M\times D^2(1),\a_M+u_0d\p_0),$$
$$(z,u_0,\p_0,u,\p,v,t) \longmapsto F(z,u_0,\p_0,u,\p,v,t)=(f(z,u_0,\p_0,u,\p), v\cdot c_f(z,u_0,\p_0,u,\p), t)$$
maps the local model $(D^\ot_\e,\eta^\ot_\e)$ to the contact germ $(D^{2n},\ker(\alpha_M+\lambda_\std))$. This concludes the proof of Theorem \ref{thm:OTxD2}.\end{proof}
\section{Weinstein cobordism from overtwisted to standard sphere} \label{sec: cobord}

The main goal of this section is proving the equivalence (1)$\Longleftrightarrow$(3)$\Longleftrightarrow$(4) in Theorem \ref{thm:main}, which is concluded in Theorems \ref{thm: ps to OT} and \ref{thm: loose to OT} . First, we state an application of the previous section which will be used in their proofs. The contact branched cover \cite[Theorem 7.5.4]{Ge2} along with Theorem \ref{thm:OTxD2} yield the following class of examples of overtwisted contact structures.

\begin{thm}\label{cor: branch cover}
Let $(Y,\xi)$ be a contact manifold and $(D,\xi|_D)$ a codimension--2 overtwisted contact submanifold. A k--fold contact branched cover of $(Y,\xi)$ along $(D,\xi|_D)$ is overtwisted for $k$ large enough.
\end{thm}

Theorem \ref{cor: branch cover} follows immediately from Theorem \ref{thm:OTxD2} since a branch cover increases the product neighborhood width of the branch locus; this latter observation has been successfully used in \cite[Section 1]{NP} for producing obstructions to symplectic fillability. In a concise manner, the reason a contact branched cover increases the size of a contact neighborhood of the branch locus is the following. Locally, the contact form near a codimension--2 submanifold $D\sse(Y,\xi)$ with trivial normal bundle can be assumed to be of the form
$$\alpha_Y=\alpha_D+r^2d\theta,$$
where $\alpha_D$ is a contact form for the contact submanifold $(D,\xi|_D)$, and we have smoothly identified $\Op(D)\cong D\times\D^2_\delta(r,\theta)$ for some $\delta\in\R^+$. In this model a $k$--fold branched cover along $D$ is given by the branched map
$$(p,\rho,\vartheta)\longmapsto (p,r,\theta)=\pi(p,\rho,\vartheta)=(p,\rho,k\vartheta),\quad (p,r,\theta)\in D\times\D^2_\delta,$$
where $(p,\rho,\vartheta)$ denote the upstairs coordinates. Thus the pull--back of the contact form is
$$\pi^*\alpha_Y=\alpha_D+(\sqrt{k}r)^2d\theta,$$
which is increasing the contact radius $r\in[0,\delta)$ to a radius of size $\rho\in[0,\sqrt{k}\delta)$, which explains Theorem \ref{cor: branch cover}.
That being said, we now apply Theorem \ref{cor: branch cover} to prove the following theorem.

\begin{thm}\label{lem:cobord}
In every dimension, there is a Weinstein cobordism $(W, \lambda, \p)$ such that the concave end $(\dd_-W,\lambda)$ is overtwisted and the convex end $(\dd_+W,\lambda)\cong (S^{2n-1},\xi_\std)$.
\end{thm}

Theorem \ref{lem:cobord} is proven assuming the equivalence (1)$\Longleftrightarrow$(2) in Theorem \ref{thm:main} which has been proven in Section \ref{sec:OTxD2}, and it also uses the inductive hypothesis in the dimension $n$. The Weinstein cobordism $(W,\lambda,\p)$ in the statement of Theorem \ref{lem:cobord} is smoothly non--trivial and it is constructed such that $\dd_-W$ is a standard smooth sphere.

\begin{proof}[Proof of Theorem \ref{lem:cobord}] Let us construct a Weinstein cobordism $(W^{2n},\lambda,\varphi)$ of finite type from an overtwisted contact structure $(S^{2n-1},\xi_{ot})$ to the standard contact sphere $(S^{2n-1},\xi_\std)$. In order to do that, consider the $A^{2n-2}_k$--Milnor fibre obtained as an $A_k$--plumbing of $k$ copies of the Weinstein manifold $(T^*S^{n-1},\la_\std,\p_\std)$, with its induced Weinstein structure. The construction of the Weinstein cobordism $(W,\la,\p)$ now has two steps.

First, we prove that the contact manifold
$$(S^{2n-1},\xi_k)=\op{OB}(A^{2n-2}_{2k-1},\tau^{-1}_1\circ\ldots\circ\tau^{-1}_{2k-1})$$
is overtwisted for $k$ large enough, and second, we construct the Weinstein cobordism to the standard contact sphere $(S^{2n-1},\xi_\std)$.

Let us first prove overtwistedness of $(S^{2n-1},\xi_k)$ for $k$ large enough. The right--veering criterion \cite{HKM} shows that $(S^3,\xi_1)$ is an overtwisted contact 3--fold, which can also be proven explicitly by finding an overtwisted 2--disk, and thus $(S^3,\xi_k)$ are overtwisted for all $k$. Now the inductive hypothesis on the dimension and the equivalence (1)$\Longleftrightarrow$(5) in Theorem \ref{thm:main} for $(2n-3)$--dimensional manifolds allows us to assume that $(S^{2n-3}, \xi_1) = \op{OB}(T^*S^{n-2}, \tau^{-1})$ is an overtwisted contact manifold.

In addition, the contact manifold $(S^{2n-3}, \xi_1)$ also admits a contact embedding into the contact manifold $(S^{2n-1},\xi_1)=\op{OB}(A^{2n-2}_1,\tau^{-1})$ compatible with the open book decomposition which corresponds to the cotangent bundle of an unknotted equatorial $S^{n-2}\subseteq S^{n-1}$. Then Theorem \ref{cor: branch cover} implies that the contact $k$--branched cover $(Y_k,\zeta_k)$ of the contact structure $(S^{2n-1},\xi_1)$ along the contact divisor $(S^{2n-3},\xi_1)$ is an overtwisted contact manifold for $k$ large enough. Note that $Y_k$ is diffeomorphic to the standard smooth sphere $S^{2n-1}$ because the smooth submanifold $S^{2n-3}$ is smoothly unknotted.

Let us now show that the contact structure $(Y_k,\zeta_k)=(S^{2n-1},\zeta_k)$ is supported by the open book decomposition $\op{OB}(A^{2n-2}_{2k-1},\tau^{-1}_1\circ\ldots\circ\tau^{-1}_{2k-1})$ and hence it is contact isotopic to $(S^{2n-1},\xi_k)$. First note that the projection map for the open book $\op{OB}(A^{2n-2}_1,\tau^{-1}_1)$ is given by argument of the map
$$f:S^{2n-1}\subset\C^{2n}\longrightarrow\C,\quad f(z_1,\ldots,z_n)=\overline{z}_1^2+\ldots+\overline{z}_n^2.$$
Then the overtwisted submanifold $(S^{2n-3},\xi_1)$ is cut out by the equation $\{z_1=0\}$ and the $k$--branched cover along it can be realized by the map $z_1\longmapsto z_1^k$. Thus the contact structure $(Y_k,\zeta_k)$ is supported by the open book induced by the argument of the map
$$f:S^{2n-1}\subset\C^{2n}\longrightarrow\C,\quad f(z_1,\ldots,z_n)=\overline{z}_1^{2k}+\overline{z}_2^2+\ldots+\overline{z}_n^2,$$
which is $\op{OB}(A^{2n-2}_{2k-1},\tau^{-1}_1\circ\ldots\circ\tau^{-1}_{2k-1})$. This proves the contactomorphism
$$(S^{2n-1},\xi_k)\cong (S^{2n-1},\zeta_k),$$
and hence the fact that $(S^{2n-1},\xi_k)$ is overtwisted for $k$ large enough.

The second step is to argue that $(S^{2n-1},\xi_k)$ is Weinstein cobordant to $(S^{2n-1},\xi_\std)$, which then constructs the required cobordism in the statement of Theorem \ref{lem:cobord} by taking $k$ large enough. Notice that by Theorem \ref{thm:NegStabCS} we have the contactomorphism
$$\op{OB}(A^{2n-2}_{2k-1}, \tau_1\circ\ldots\circ\tau_{2k-1}) = (S^{2n-1}, \xi_\std),$$
since this open book is just the trivial open book $(D^{2n-2}, \op{id})$ positively stabilized $(2k-1)$ times. Then we can perform two Weinstein handle attachments as described in Proposition \ref{prop: ob surgery} to each zero section in the Weinstein page $A^{2n-2}_{2k-1}$, giving a total of $(4k-2)$ critical handle attachments, which construct the Weinstein cobordism from $(S^{2n-1}, \xi_k)$ to $(S^{2n-1}, \xi_\std)$. \end{proof}

The following proposition is the remaining ingredient before we are able to conclude the equivalence (1)$\Longleftrightarrow$(3) from the above Theorem \ref{lem:cobord}.

\begin{prop} \label{prop: loose flex}
Let $(Y, \xi)$ be a contact manifold, and suppose that the standard Legendrian unknot is a loose Legendrian submanifold in $(Y, \xi)$. Let $(W, \lambda, \p)$ be an arbitrary Weinstein cobordism and let $SY$ be the symplectization of $Y$. Then connected sum cobordism $W \ol{\#} SY$ is always a flexible Weinstein cobordism.
\end{prop}

\begin{proof}
Since the symplectization $SY$ is a Weinstein trivial product, the critical points of the Weinstein cobordism $W \ol{\#} SY$ are the same as the critical points of the cobordism $(W,\la,\p)$. Let $p$ be a critical point in $W \ol{\#} SY$ of index $n$, $M = \p^{-1}(\p(c) - \e)$ a level set of $W$, and $\Lambda \sse M \# Y$ the Legendrian attaching sphere of the critical point $p$.

Let $\Lambda_0 \sse (Y,\xi)$ be the standard Legendrian unknot, and $U \sse M \# Y$ the union of a Darboux chart containing $\Lambda_0$ and a loose chart for $\Lambda_0$. Since $\Lambda_0$ is loose as a Legendrian in $(Y,\xi)$ and the Legendrian $\Lambda$ is the descending sphere of a critical point of $(W,\la,\p)$, we know that $\Lambda$ is disjoint from $U$, and therefore their Legendrian connected sum $\Lambda \# \Lambda_0$ is a loose Legendrian, even with the same loose chart as $\Lambda_0$. Then, the Legendrian connected sum $\Lambda \# \Lambda_0$ is Legendrian isotopic to $\Lambda$ since the Legendrian $\Lambda_0$ is the standard Legendrian unknot, and thus follows that the Legendrian $\Lambda$ is loose.
\end{proof}

This allows us to prove the equivalence (1)$\Longleftrightarrow$(3), which is the following theorem.

\begin{thm} \label{thm: ps to OT}
Let $\Lambda_0$ be the standard Legendrian unknot inside a contact manifold $(Y, \xi)$. If $\Lambda_0$ is a loose Legendrian then $(Y, \xi)$ is overtwisted.
\end{thm}

\begin{proof} Let $(W,\la,\p)$ be the cobordism constructed in Theorem \ref{lem:cobord}, and apply Proposition \ref{prop: loose flex} to conclude that the vertical connected sum $W \ol{\#} SY$ is a flexible Weinstein cobordism. The concave end of the cobordism $\dd_-(W \ol{\#} SY) = \dd_- W \# Y$ is an overtwisted contact manifold since the contact boundary $\dd_- W$ is overtwisted itself, and thus Proposition \ref{prop: flexible ot} implies that the contact convex end
$$\dd_+(W \ol{\#} SY) = (S^{2n-1}, \xi_\std) \# (Y, \xi) \cong (Y, \xi)$$
is an overtwisted contact manifold as well.\end{proof}

Theorem \ref{thm: ps to OT} also implies the equivalence (1)$\Longleftrightarrow$(4). Indeed, the standard unknot in $(Y, \xi)$ is defined by the inclusion of a small Darboux chart in $(Y,\xi)$ and thus if the contact manifold contains a small plastikstufe with trivial rotation, the unknot must be in the complement. Therefore, Theorems \ref{thm: PS to loose} and \ref{thm: ps to OT} imply

\begin{thm} \label{thm: loose to OT}
Let $(Y,\xi)$ be a contact manifold containing a small plastikstufe with spherical core and trivial rotation. Then $(Y, \xi)$ is overtwisted.
\end{thm}

Thus far in the article we have proven the equivalences (1)$\Longleftrightarrow$(2)$\Longleftrightarrow$(3)$\Longleftrightarrow$(4) in Theorem \ref{thm:main}. The following two sections are respectively dedicated to the proofs of the two remaining equivalences, that is, the characterization in terms of surgeries (1)$\Longleftrightarrow$(5), and the criterion in terms of open book decompositions (1)$\Longleftrightarrow$(6).

\section{(+1)--surgery on loose Legendrians} \label{sec: surgery}

In this section we prove the equivalence (1)$\Longleftrightarrow$(5) in Theorem \ref{thm:main} by using the characterization given by Theorem \ref{thm: loose to OT}.  For our purpose, we use the following model of contact $(+1)$--surgery on a Legendrian sphere, defined in the article \cite[Section 9]{Avdek}.

Let $\Lambda \sse (Y, \xi)$ be a Legendrian sphere in a contact manifold. A neighborhood of the Legendrian $\Lambda$ can be identified with a neighborhood of the zero section in the first--jet space
$$(\J^1(S^{n-1}),\ker\a_\std) = (T^*S^{n-1} \x \R(z),\ker(dz-\la_\std)).$$

Consider the smooth manifold $Y'$ obtained by removing the piece $D^*S^{n-1} \x (0,1)$ from $Y$, and then gluing the boundary to itself with the identification $(x, 0) \sim (\tau^{-1}(x), 1)$ and $(x, t) \sim (x, t')$ for $x \in \dd D^*S^{n-1}$, where $\tau: T^*S^{n-1}\longrightarrow T^*S^{n-1}$ denotes the Dehn twist along a zero section \cite{Se97}. Note that $Y'$ is smooth manifold since the diffeomorphism $\tau$ is compactly supported, and it has a canonical contact structure $\xi'$ because the gluing diffeomorphism $\tau$ is a symplectomorphism.

\begin{definition}\label{def:sur}
The contact manifold $(Y', \xi')$ obtained with the procedure above is said to be the \emph{contact $(+1)$--surgery of $(Y, \xi)$ along $\Lambda$}.\hfill$\Box$
\end{definition}

\begin{remark}
Given a Legendrian sphere $\Lambda \sse (Y, \xi)$, the contactomorphism type of the contact surgery $(Y', \xi')$ depends on the chosen parametrization $f: S^{n-1} \longrightarrow \Lambda$ of the Legendrian submanifold. In fact \cite[Theorem A]{DRE} shows that the class $[\tau] \in \pi_0\op{Symp}(T^*S^{n-1})$ genuinely depends on this parametrization. However, in our context we are able to dismiss this technical distinction since any two parametrizations of \emph{loose} Legendrian spheres are ambiently contact isotopic \cite[Theorem 1.2]{loose}.\hfill$\Box$
\end{remark}

\begin{remark}
Since the symplectomorphism $\tau$ does not preserve the Liouville form the gluing above should be technically performed in the region of the contactization given by $\{0 \leq z \leq f\} \sse T^*S^{n-1} \x \R$, where $f \in C^\infty(T^*S^n)$ is a positive primitive of $\lambda - \tau^*\lambda$.\hfill$\Box$
\end{remark}

In this surgery model introduced in Definition \ref{def:sur}, we can prove the equivalence (4)$\Longleftrightarrow$(5) in Theorem \ref{thm:main}, which also establishes \cite[Conjecture 9.16]{Avdek}.

\begin{thm} \label{thm: down surgery}
Let $\Lambda \sse (Y, \xi)$ be a loose Legendrian submanifold. Then the contact $(+1)$--surgery of $(Y,\xi)$ along $\Lambda$ contains a small plastikstufe with spherical core and trivial rotation.
\end{thm}

\begin{figure}[h!]
\centering
  \includegraphics[scale=0.5]{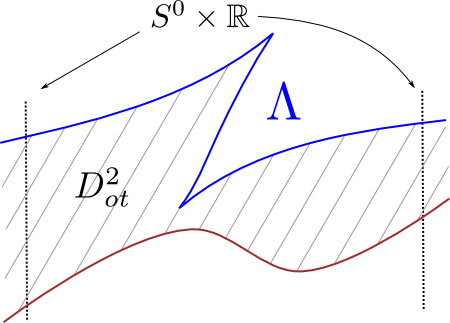}
  \caption{The overtwisted disk inside $(M_x, \xi')$. Here we are viewing $M_x$ as presented by surgery in the front projection of $\J^1(S^1)$. In particular, the transverse curve $S^0 \x \R = \{p = 0, q \in S^0\}$. Since the boundary of $D^2_\ot$ and the surgery curve $\Lambda$ both have positive slope in the front, we can choose $D^2_\ot$ to lie in the region $p>0$, in particular making it disjoint from $S^0 \x \R$.}
  \label{fig:surgery OT disk}
\end{figure}

\begin{proof}
Since the Legendrian sphere $\Lambda$ is loose, we can choose a Legendrian sphere $\wt{\Lambda}$ whose spherical stabilization gives the Legendrian $\Lambda$ \cite{loose}. Choose coordinates in a neighborhood of the Legendrian $\wt\Lambda$ identifying it with a neighborhood of the zero section in the jet space $(T^*S^{n-1} \x \R,\ker\alpha_\std)$, and we can then represent the original Legendrian $\Lambda$ as the zero section stabilized over the equator $S^{n-2} \subseteq S^{n-1}$. For a fixed point $x \in S^{n-2}$ in the equator, define the circle $S^1_x\subseteq S^{n-1}$ to be the unique meridian passing through the point $x$ and the north and south poles, and consider the submanifold $\J^1(S^1_x) \sse T^*S^{n-1} \x \R$. The jet space $\J^1(S^1_x)$ is a $3$-dimensional contact submanifold contactomorphic to $T^*S^1 \x \R$, and under this contactomorphism the intersection $\Lambda \cap \J^1(S^1_x)$ is given as the stabilization of the zero section. Note also that for $x \neq y$, we can identify $\J^1(S^1_x) \cap \J^1(S^1_y) \cong S^0 \x \R$ where $S^0$ is the union of the north and south poles.

Because the Dehn twist $\tau:T^*S^{n-1}\longrightarrow T^*S^{n-1}$, which is used to perform the contact surgery, is a symplectomorphism defined using the geodesic flow on the sphere and the meridian $S^1_x$ is a geodesic submanifold, it necessarily preserves the submanifold $T^*S^1_x$. Now, if we let
$$q:(Y \sm \Op(\Lambda),\xi)\longrightarrow (Y',\xi')$$
be the quotient map realizing the contact $(+1)$--surgery on $\Lambda$, the image $q(\J^1(S^1_x))$ is a contact submanifold $M_x$ which is itself contactomorphic to the contact $(+1)$--surgery of the 1--jet space $\J^1(S^1_x)$ along the stabilized Legendrian $\Lambda\cap\J^1(S^1_x)$. Then the contact manifold $(M_x,\xi')$ is overtwisted for every $x\in S^{n-2}$, even in the complement of the submanifold $S^0 \x \R$. See \cite[Theorem 1.2]{DGS} and \cite[Exercise 11.2.10]{OS} for details on an overtwisted disk for $(M_x,\xi')$, and see Figure 7 for a schematic depiction. The entire picture is symmetric about $x\in S^{n-2}$, and thus the construction defines a plastikstufe $\mathcal{P}$ with spherical core.

It remains to show that this plastikstufe $\mathcal{P}$ has trivial rotation class and that it is contained in a smooth ball. We prove these claims simultaneously by showing that an open leaf of $\mathcal{P}$ is contained in a Legendrian disk. Indeed, an open leaf of $\mathcal{P}$ is given as the union of Legendrian arcs in $M_x$ and we can consider an isotopy between this arc and a small Legendrian arc in $\Lambda\cap\J^1(S^1_x)$ disjoint from the two vertical lines $S^0 \x \R$. Then by considering this symmetrically with respect to the point $x \in S^{n-2}$, we get an isotopy from an open leaf of $\mathcal{P}$ to an annulus $S^{n-2} \x [0,1] \sse \Lambda$, and since the Legendrian $\Lambda$ is a sphere this annulus extends to a Legendrian disk inside the Legendrian $\Lambda$.
\end{proof}
This concludes the equivalence (1)$\Longleftrightarrow$(5) in Theorem \ref{thm:main}. This equivalence already suffices to prove the two applications Proposition \ref{prop: concord} and Corollary \ref{thm: weinstein exist} on the existence of Weinstein cobordisms with an overtwisted concave end, which we explain in Section \ref{ssec:cons}. However, we follow the natural order and proceed with the remaining equivalence in the statement of Theorem \ref{thm:main}.

\section{Stabilization of Legendrians and open books} \label{sec:NegStab}
In this section we conclude the proof of Theorem \ref{thm:main}, by proving the equivalence (3)$\Longleftrightarrow$(6). To do this, in Subsection \ref{ssec:looseob} we will relate two known procedures in contact topology: the stabilization of a Legendrian submanifold and the stabilizations of a compatible open book. The link between these two procedures can be established through Lagrangian surgery \cite{Po}, also referred to as Polterovich surgery, the details of which are first explained in Subsection \ref{ssec:cone/cusp}. The results in Subsections \ref{ssec:cone/cusp} and \ref{ssec:looseob} imply the following result.

\begin{thm} \label{thm: negOB to loose}
Let $(S^{2n-1}, \xi_-)$ be the contact manifold supported by the open book whose page is $(T^*S^{n-1},\lambda_\std)$ and whose monodromy is the left handed Dehn twist along the zero section. Then the standard Legendrian unknot in $(S^{2n-1}, \xi_-)$ is loose.
\end{thm}

In light of Theorem \ref{thm:NegStabCS}, Theorem \ref{thm: negOB to loose} implies (3)$\Longleftrightarrow$(6) and thus Theorem \ref{thm:main}. Indeed, the fact that any overtwisted contact manifold admits a negatively stabilized open book follows quickly from known results as we now explain.

Let $(Y, \xi)$ be an overtwisted contact structure, and note that the set of almost contact structures on the sphere forms a group under connected sum \cite[Chapter 4.3]{Ha}. Now the existence $h$--principle \cite[Theorem 1.2]{BEM} implies that there is an overtwisted contact structure $(Y, \eta)$ such that the contact connected sum $(Y \# S^{2n-1}, \eta \# \xi_-)$ is in the same homotopy class of almost contact structures as the given contact manifold $(Y, \xi)$, and since the contact structures $\xi$ and $\eta \# \xi_-$ are both overtwisted, they are necessarily isotopic. Now E.~Giroux's existence Theorem \ref{ob exist} states that the contact structure $(Y, \eta)$ is compatible with an open book $(W,\la,\p)$ and, by using his Theorem \ref{thm:NegStabCS}, the negative stabilization of the open book $(W,\la,\p)$ supports the contact structure $(Y,\eta \# \xi_-)$, which is isotopic to $(Y,\xi)$. This shows the implication (1)$\Rightarrow$(6), and therefore Theorem \ref{thm: negOB to loose} is the main remaining ingredient in order to prove the equivalence (1)$\Longleftrightarrow$(6). Let us then move towards the proof of Theorem \ref{thm: negOB to loose}.

\subsection{Legendrians in open books}
In order to prove Theorem \ref{thm: negOB to loose}, we develop some combinatorics for describing Legendrian submanifolds in adapted open books decompositions.

Let $(Y, \xi) = \op{OB}(W, \lambda, \p)$, and recall that if $L \sse (W,\lambda)$ is an exact Lagrangian, it determines a Legendrian $\Lambda \sse Y$ as noted in Subsection \ref{ssec: OB intro}. The relationship was denoted by the equality $(Y, \xi, \Lambda) = \op{OB}(W, \lambda, \p, L)$, and we emphasize that the Legendrian $\op{OB}(W, \lambda, \p, L)$ is contactomorphic to the Legendrian defined by $\op{OB}(W, \lambda, \psi \circ \p \circ \psi^{-1}, \psi(L))$, and typically distinct from the Legendrian defined by $\op{OB}(W, \lambda, \psi \circ \p \circ \psi^{-1}, L)$. In particular, the Legendrian $\op{OB}(W, \lambda, \p, L)$ is contactomorphic to $\op{OB}(W, \lambda, \p, \p(L))$.

These observations are relevant to the proof and understanding of Theorem \ref{thm: negOB to loose}. The next subsection contains the results expressing Lagrangian surgery on two Lagrangians in terms of Legendrian connected sums of their Legendrian lifts.
\subsection{Lagrangian Surgery and Legendrian Sums} \label{ssec:cone/cusp}

The Dehn--Seidel twists \cite{Se97}[Chapter I.2] along exact Lagrangian spheres are an important class of compactly supported exact symplectomorphisms of a Liouville domain $(W,\lambda)$. Given a contact manifold, an adapted open book decomposition precisely consists of a Liouville domain, the page, and a symplectic monodromy, which oftentimes consists of Dehn--Seidel twists. From this viewpoint, it is relevant for the study of contact topology to reinterpret the action of Dehn twists on Lagrangians in terms of their Legendrian lifts. This is the aim of this subsection.

We focus on the case where $L \sse (W,\lambda)$ is an exact Lagrangian and $S \sse W$ is a Lagrangian sphere transversely intersecting $L$ in one point. In this case, the Dehn twist of $L$ around $S$ can be interpreted as the Polterovich surgery \cite{FOOO,Po} of $L$ and $S$, denoted by $L+S$. The definition and details of the Polterovich surgery will be given momentarily, after Remark \ref{rmk:Reeb} below. For now, we state its relation to Dehn twists:

\begin{thm}[\cite{Se99}]\label{thm:dehn/sur}
The Lagrangian surgery $S+L$ is Lagrangian isotopic to $\tau_S(L)$.\\
The Lagrangian surgery $L+S$ is Lagrangian isotopic to $\tau^{-1}_S(L)$.
\end{thm}

\begin{figure}[h!]
\centering
\begin{subfigure}{.5\textwidth}
  \centering
  \includegraphics[scale=0.5]{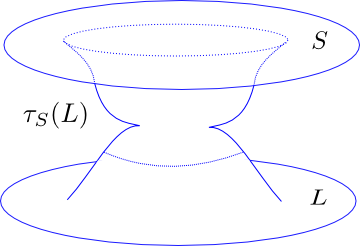}
  \caption{The Legendrian lift of $\tau_S(L)$.}
\end{subfigure}%
\begin{subfigure}{.5\textwidth}
  \centering
  \includegraphics[scale=0.5]{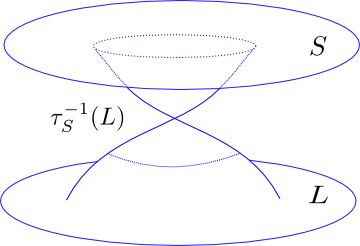}
  \caption{The Legendrian lift of $\tau^{-1}_S(L)$.}
\end{subfigure}
\caption{The statement of Theorem \ref{thm:sur/sum}.}
\label{fig:cuspcone}
\end{figure}

We now model this operation in terms of the fronts of Legendrian lifts $\Lambda$ and $\Sigma$ of the exact Lagrangians $L$ and $S$. The main technical result in this section is the following theorem.

\begin{thm}\label{thm:sur/sum}
Let $L,S\subseteq(W,\lambda)$ be two exact Lagrangians transversely intersecting at a point $p=L\cap S$, and consider the contactization $(Y,\xi)=(W\times\R(z),\ker\{dz-\la\})$ of $(W,\la)$.

There exists a Darboux chart in $(Y,\xi)$ centered at $p\in(W,\lambda)$ such that the front projection of the Legendrian lift of $S+L$ is as depicted in Figure \ref{fig:cuspcone}.A.

There exists a Darboux chart in $(Y,\xi)$ centered at $p\in(W,\lambda)$ such that the front projection of the Legendrian lift of $L+S$ is as depicted in Figure \ref{fig:cuspcone}.B.
\end{thm}

\begin{remark}\label{rmk:Reeb}
Figure \ref{fig:cuspcone} depicts the following situation. The lower horizontal sheet is the lift of a Lagrangian disk $D_L$ contained in the exact Lagrangian $L$ centered at $p$, whereas the upper horizontal sheet is the lift of a Lagrangian disk $D_S$ contained is $S$ also centered at $p$.

Note that there exists a unique Reeb chord connecting the Legendrian lifts of the Lagrangians disks $D_L$ and $D_S$, corresponding to the intersection point $p=L\cap S$ in the Lagrangian projection. Then Figures \ref{fig:cuspcone}.A and B are obtained by respectively substituting this unique local Reeb chord by either a rotationally symmetric cusp or the rotationally symmetric cone. The Legendrian isotopy class of the fronts in Figure \ref{fig:cuspcone}.A and \ref{fig:cuspcone}.B are respectively referred to as the cusp-sum and cone-sum, or the cusp and the cone, of $\Lambda$ and $\Sigma$ along the Reeb chord over the intersection point $p=L\cap S$.\hfill$\Box$
\end{remark}

Let us now review L.~Polterovich's Lagrangian surgery \cite{Po} and prove Theorem \ref{thm:sur/sum}.

Consider local coordinates $(q_1,\ldots,q_{n-1},p_1,\ldots,p_{n-1})\in\R^{2n-2}$ such that the Lagrangians $L$ and $S$ are locally expressed as $L=\{p_1=0,\ldots,p_{n-1}=0\}$, $S=\{q_1=p_1,\ldots,q_{n-1}=p_{n-1}\}$ and the Liouville form reads
$$\displaystyle\lambda_\std=\sum_{i=1}^{n-1}p_idq_i.$$

The Lagrangian surgeries $S+L$ and $L+S$ are respectively described in terms of two Lagrangian handles $\Gamma^\pm$ \cite{Po}. These Lagrangian handles are depicted in Figure \ref{fig:LagHand3}, and in order to parametrize them we use coordinates $t=(t_1,\ldots,t_{n-1})\in\R^{n-1}$.\\

\begin{figure}[h!]
\centering
\begin{subfigure}{.5\textwidth}
  \centering
  \includegraphics[scale=0.35]{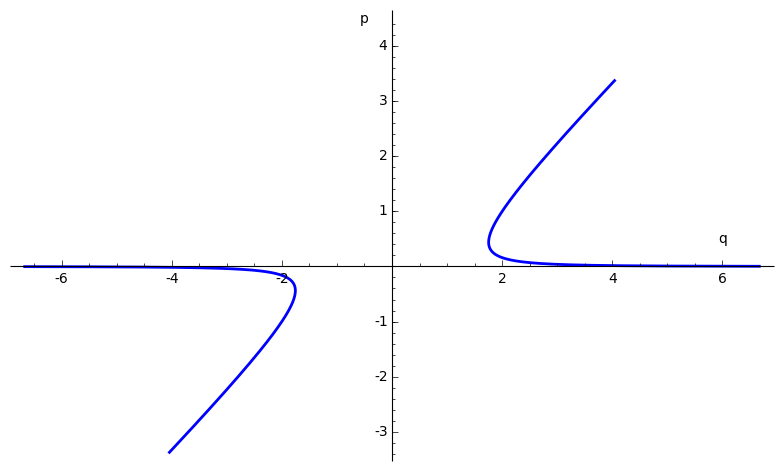}
  \caption{The Lagrangian handle $\Gamma^+$.}
\end{subfigure}%
\begin{subfigure}{.5\textwidth}
  \centering
  \includegraphics[scale=0.35]{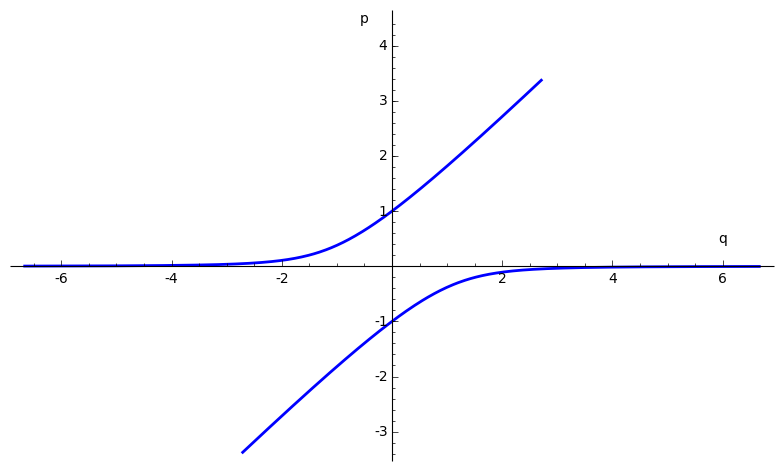}
  \caption{The Lagrangian handle $\Gamma^-$.}
\end{subfigure}
\caption{The Lagrangian handles $\Gamma^{\pm}\subseteq\R^{2n-2}(q,p)$.}
\label{fig:LagHand3}
\end{figure}

First, we consider the case of the positive Lagrangian handle $\Gamma^+$; it can be described via the parametrization $\Gamma^+:\R^{n-1}\setminus\{0\}\longrightarrow\R^{2n-2}$
defined as
$$\Gamma^+(t_1,\ldots,t_{n-1})=\left((\mu+\mu^{-1})t_1,\ldots,(\mu+\mu^{-1})t_{n-1},\mu t_1,\ldots,\mu t_{n-1}\right)\mbox{ where }\displaystyle\mu=\sum_{i=1}^{n-1}t_i^2.$$

Note that we have the two asymptotics $\displaystyle\lim_{\mu\to\infty}\Gamma^+\subseteq S$ and $\displaystyle\lim_{\mu\to0}\Gamma^+\subseteq L$.
By definition, the Polterovich surgery $S+L$ is obtained by gluing the above positive Lagrangian handle $\Gamma^+$ to the Lagrangian $L$ at the limit $\mu=0$, and to the Lagrangian $S$ at the limit $\mu=\infty$.

Analogously, the Polterovich surgery $L+S$ is obtained by using the negative Lagrangian handle $\Gamma^-:\R^{n-1}\setminus\{0\}\longrightarrow\R^{2n-2}$ parametrized by
$$\Gamma^-(t_1,\ldots,t_{n-1})=\left((\mu-\mu^{-1})t_1,\ldots,(\mu-\mu^{-1})t_{n-1},\mu t_1,\ldots,\mu t_{n-1}\right).$$
This parametrization satisfies the asymptotics $\displaystyle\lim_{\mu\to\infty}\Gamma^-\subseteq S$ and $\displaystyle\lim_{\mu\to0}\Gamma^-\subseteq L$, and can be glued to $L$ and $S$ in the asymptotic limits, thus constructing the Lagrangian $L+S$.

\begin{remark}
The Lagrangian handles $\Gamma^\pm$ can be parametrized to be not only asymptotic to $L$ and $S$ but actually coincide with them in the local model. This is a matter of introducing the appropriate cut--off functions, and the Lagrangian isotopy type of the construction remains unchanged.\hfill$\Box$
\end{remark}

\begin{proof}[Proof of Theorem \ref{thm:sur/sum}] In the contactization $(\R^{2n-1}(q,p;z),\ker(dz-\lambda_\std))$ of the standard exact Weinstein manifold $(\R^{2n-2}(q,p),\lambda_\std)$, the Lagrangian $L$ described above lifts to the Legendrian
$$\Lambda=\{(q_1,\ldots,q_{n-1},0,\ldots,0;0)\}$$
and the Lagrangian $S$ lifts to the Legendrian
$$\displaystyle\Sigma=\{(q_1,\ldots,q_{n-1},q_1,\ldots,q_{n-1};(q_1^2+\ldots+q_{n-1}^2)/2)\}.$$

We can lift the exact Lagrangian $\Gamma^+$ to the contactization via $z=z(t_1,\ldots,t_{n-1})$:
$$dz(t)=\sum_{i=1}^{n-1}\mu t_i d\left((\mu+\mu^{-1})t_i\right)=\sum_{i=1}^{n-1}(\mu^2+1)t_idt_i+\sum_{i=1}^{n-1}\mu t^2_i(1-\mu^{-2})d\mu=$$
$$=\sum_{i=1}^{n-1}(\mu^2+1)t_idt_i+(\mu^2-1)d\mu$$
Hence the partial derivatives of $z(t)$ are:
$$\frac{\partial z(t)}{\partial t_i}=(\mu^2+1)t_i+(\mu^2-1)2t_i=(3\mu^2-1)t_i.$$
Thus the $z$--coordinate of the lift is parametrized by $z(t)=\frac{1}{2}(\mu^3-\mu)$ and in the front projection $\R^{n}(q_1,\ldots,q_{n-1},z)$ we obtain a rotationally symmetric cusp. Part of the front projections in dimensions 3 and 5 are depicted in Figures \ref{fig:cusp3} and \ref{fig:cusp5}.
\begin{center}
\begin{figure}[h!]
\includegraphics[scale=0.5]{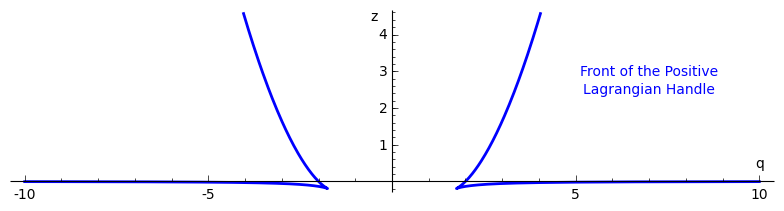}
\caption{Front projection to $\R^2(q_1,z)$ of the Legendrian lift of the positive Lagrangian handle $\Gamma^+\subseteq\R^3(q_1,p_1,z)$ for $t\in[-1.5,-0.1]\cup[0.1,1.5]$.}
\label{fig:cusp3}
\end{figure}
\end{center}

\begin{center}
\begin{figure}[h!]
\includegraphics[scale=0.5]{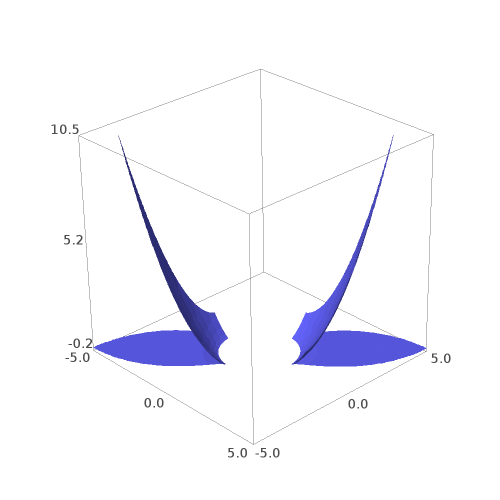}
\caption{Front projection to $\R^3(q_1,q_2,z)$ of the Legendrian lift of $\Gamma^+\subseteq\R^5$ with $(t_1,t_2)$ in the range $[-1.2,-0.1]\times[-1.2,-0.1]\cup[0.1,1.2]\times[0.1,1.2]$.}
\label{fig:cusp5}
\end{figure}
\end{center}
This describes the Polterovich surgery $S+L$ in terms of the cusp-sum of the two Legendrians $\Lambda$ and $\Sigma$ respectively lifting $L$ and $S$, and concludes the first statement of Theorem \ref{thm:sur/sum}.

Regarding the Legendrian lift of the Polterovich surgery $L+S$, the $z$--coordinate of the lift to the contactization satisfies
$$dz(t)=\sum_{i=1}^{n-1}\mu t_i d\left((\mu-\mu^{-1})t_i\right)=\sum_{i=1}^{n-1}(\mu^2-1)t_idt_i+(\mu^2+1)d\mu.$$
Thus we conclude that the partial derivatives of $z(t)$ are given by
$$\frac{\partial z(t)}{\partial t_i}=(3\mu^2+1)t_idt_i$$
and $z(t)=\frac{1}{2}(\mu^3+\mu)$ provides a lift for $\Gamma^-$. The front projection is depicted in Figures \ref{fig:cone3} and \ref{fig:cone5} in the 3--dimensional and 5--dimensional cases. 
\begin{center}
\begin{figure}[h!]
\includegraphics[scale=0.5]{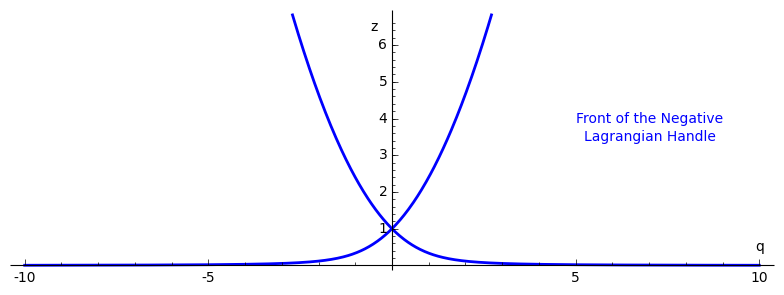}
\caption{Front projection to $\R^2(q_1,z)$ of the Legendrian lift of the handle $\Gamma^-\subseteq\R^3(q_1,p_1,z)$ with $t\in[-1.5,-0.1]\cup[0.1,1.5]$.}
\label{fig:cone3}
\end{figure}
\end{center}

\begin{center}
\begin{figure}[h!]
\includegraphics[scale=0.5]{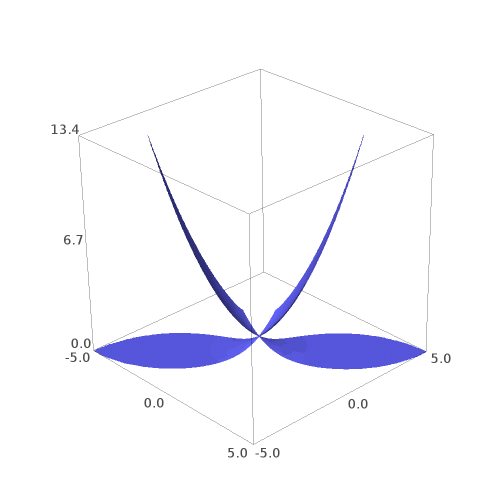}
\caption{Front projection to $\R^3(q_1,q_2,z)$ of the Legendrian lift of $\Gamma^-\subseteq\R^5$ with parameters $(t_1,t_2)\in[-1.2,-0.1]\times[-1.2,-0.1]\cup[0.1,1.2]\times[0.1,1.2]$.}
\label{fig:cone5}
\end{figure}
\end{center}
This concludes the second statement of Theorem \ref{thm:sur/sum}.
\end{proof}
\subsection{Loose Legendrians in open books}\label{ssec:looseob} In order to show that the Legendrian unknot in the contact manifold $(S^{2n-1},\xi_-)=\op{OB}(T^*S^{n-1},\tau^{-1})$ is a loose Legendrian submanifold, we need an understanding of looseness and the standard unknot in the open book framework. This is the content of Propositions \ref{lem: uk ob} and \ref{lem: stab ob}, which we use in order to prove Theorem \ref{thm: negOB to loose}.

\begin{prop}\label{lem: uk ob}
Let $(Y,\xi)=\op{OB}(W, \lambda, \p)$ be a contact manifold and $(W \cup H,\lambda, \p \circ \tau_S)$ a positive stabilization. The Legendrian lift of $S$ to $(Y, \xi)$ is the standard unknot.
\end{prop}
Proposition \ref{lem: uk ob} can be deduced from the theory of Lagrangian vanishing cycles \cite[Chapter III]{Se08} and their Lagrangian vanishing thimbles.

\begin{prop} \label{lem: stab ob}
Let $(W \cup H, \lambda, \p \circ \tau_S)$ be a positively stabilized open book and $L\sse W$ an exact Lagrangian which transversely intersects $S$ in one point. Then the Legendrian $(W \cup H, \lambda, \p \circ \tau_S, L)$ is contactomorphic to the Legendrian $(W \cup H, \lambda, \p \circ \tau_S, \tau^{-1}_S(L))$ and the Legendrian $(W \cup H, \lambda, \p \circ \tau_S, \tau_S(L))$ is loose.
\end{prop}

\begin{proof}
Choose a Legendrian lift $\Lambda$ for the Lagrangian $L$ which has angle $\theta = 0$ at the intersection point $L \cap S$, and a Legendrian lift for $S$ with angle $\theta = \e$ for a small constant $\e\in\R^+$. Theorem \ref{thm:sur/sum} implies that the Legendrian lifts of $\tau_S(L)$ and $\tau_S^{-1}(L)$ are represented by the cusp and cone sums Legendrian fronts. Indeed, since they intersect in one point, we know by Theorem \ref{thm:dehn/sur} that $\tau_S(L) = S+L$ and $\tau_S^{-1}L = L+S$. Then by Theorem \ref{thm:sur/sum} the Legendrian lift of $S+L$ corresponds to the cusp-sum, and the Legendrian lift of $L+S$ corresponds to their cone-sum. Note that the Legendrian lift of $S$ is the Legendrian unknot contained in a Darboux ball which is disjoint from the Legendrian $\Lambda$, and since any two Darboux balls are contact isotopic we have that cone or cusp summing with the unknot is a local operation on the Legendrian $\Lambda$. Let us now discuss the two cases.

For the Legendrian $\op{OB}(W \cup H, \lambda, \p \circ \tau_S, \tau^{-1}_S(L))$, we note that cone-summing a Legendrian with a small Legendrian unknot does not change the Legendrian isotopy type since this is just the $S^{n-2}$--spinning of the first Legendrian Reidemeister move. Therefore Legendrian $\Lambda$ is Legendrian isotopic to the Legendrian lift of the exact Lagrangian $L+S = \tau^{-1}_S(L)$.

In contrast, the situation is different for the Legendrian $\op{OB}(W \cup H, \lambda, \p \circ \tau_S, \tau_S(L))$. Indeed, observe that the cusp-sum of a Legendrian submanifold with a small Legendrian unknot explicitly creates a loose chart \cite{CE,loose} and therefore the Legendrian lift of the exact Lagrangian $\tau_S(L) = S+L$ is actually a loose Legendrian.
\end{proof}

Propositions \ref{lem: uk ob} and \ref{lem: stab ob} are the ingredients needed to prove Theorem \ref{thm: negOB to loose}.

\subsection{Proof of Theorem \ref{thm: negOB to loose}}\label{ssec:1=5} Consider the contact manifold
$$(S^{2n-1},\xi_-)=\op{OB}(T^*L,\la_\std;\tau_L^{-1})$$
obtained by negatively stabilizing the contact open book $(S^{2n-1},\xi_\st)=\op{OB}(D^{2n-2},\la_\st;\mbox{id})$, where we have denoted $L\cong S^{n-1}$ for the zero section of the stabilized Weinstein page. Let us choose a cotangent fiber in the Weinstein page $(T^*L,\la_\std)$ and positively stabilize the compatible open book above along this cotangent fiber. The Weinstein page $(W, \lambda) = T^*S^{n-1} \cup H$ of the resulting open book is a plumbing of two copies of the Weinstein structure $(T^*S^{n-1},\la_\std)$ whose exact Lagrangian zero sections $L$ and $S$ intersect in one point.

First, the Legendrian $\Lambda_0 = \op{OB}(W, \lambda,\tau_L^{-1}\circ\tau_S, S)$ is the standard Legendrian unknot by Proposition \ref{lem: uk ob}. And second, the Legendrian submanifold $\Lambda_\ell = \op{OB}(W, \lambda, \tau_L^{-1}\circ\tau_S, \tau_S(L))$ is a loose Legendrian by Proposition \ref{lem: stab ob}. In consequence, suffices to show that these two Legendrians are contactomorphic, which follows from the fact that the Legendrian $\Lambda_0$ is contactomorphic to the Legendrian
$$\op{OB}(W, \lambda, \tau_L^{-1}\circ\tau_S, (\tau^{-1}_L\circ\tau_S)(S))$$
and the exact Lagrangian isotopy $(\tau^{-1}_L\circ\tau_S)(S) = \tau^{-1}_L(S) = S+L = \tau_S(L)$.\hfill$\Box$

\section{Proof of Theorem \ref{thm:main}}\label{sec:mainproof} 

In this section we formally prove Theorem \ref{thm:main} using the results in Sections \ref{sec:OTxD2}, \ref{sec: cobord}, \ref{sec: surgery}, and Section \ref{sec:NegStab}. First, the $h$-principle \cite[Theorem 1.2]{BEM} directly gives the implications $(1)\Rightarrow(2)$ and $(1)\Rightarrow(4)$. The implication $(1)\Rightarrow(3)$ also follows directly from \cite[Theorem 1.2]{BEM}, or alternatively using \cite[Theorem 1.1]{MNPS}, which states $(4)\Rightarrow(3)$. The same $h$-principle \cite[Theorem 1.2]{BEM} gives the implication $(1)\Rightarrow(6)$, as explained in Section \ref{sec:NegStab} right after the statement of Theorem \ref{thm: negOB to loose}. Finally, the implication $(1)\Rightarrow(5)$ follows from the implication $(6)\Rightarrow(5)$, which itself follows from the relation between Dehn twists in the symplectic monodromy of an adapted open book and contact surgeries, see for instance \cite[Theorem 4.4]{Ko} and \cite[Section 3]{CM2}.

By the above paragraph, the implications $(1)\Rightarrow(2),(3),(4),(5),(6)$ hold. Let us now use the results in this article to conclude the converse. Indeed, Theorem \ref{thm:OTxD2} shows $(2)\Rightarrow(1)$. The implication  $(3)\Rightarrow(1)$ is the content of Theorem \ref{thm: loose to OT}. The implication $(4)\Rightarrow(1)$ follows from the now proven implication  $(3)\Rightarrow(1)$ and  $(4)\Rightarrow(3)$, which holds by \cite[Theorem 1.1]{MNPS}. The implication $(5)\Rightarrow(1)$ follows Theorem \ref{thm: down surgery}, which proves $(5)\Rightarrow(4)$ and the implication $(4)\Rightarrow(1)$. Finally, $(6)\Rightarrow(1)$ follows from Theorem \ref{thm: negOB to loose}, which proves $(6)\Rightarrow(3)$, and Theorem \ref{thm: loose to OT}, which shows $(3)\Rightarrow(1)$.\hfill$\Box$

Let us now provide two applications of Theorem \ref{thm:main} to contact topology.

\section{Applications}\label{ssec:cons} 

In this section we explore consequences of Theorem \ref{thm:main}. Subsection \ref{sec:sizes} discusses neighborhoods in contact topology in relation to Theorem \ref{thm:main} and Subsection \ref{ssec:conc} constructs a Weinstein concordance between an overtwisted contact structure on the $(2n-1)$--dimensional sphere and the standard contact structure $(S^{2n-1},\xi_\std)$.

\subsection{Neighborhood size and contact squeezing}\label{sec:sizes} 
Theorem \ref{thm:main} emphasizes in its first equivalence $1=2$ the importance of the size of a neighborhood of a contact submanifold. In this direction it is relevant to understand the dichotomy between tight and overtwisted contact structures in terms of small and large neighborhoods.

\begin{thm} \label{cor: OT squeezing}
Let $(Y, \ker\alpha)$ be an overtwisted contact manifold. There exists a radius $R_0\in\R^+$ such that for any $R > R_0$, there exists a compactly supported contact isotopy
$$f_t:(Y \x \C, \ker\{\alpha + \lambda_\std\})\longrightarrow(Y \x \C, \ker\{\alpha + \lambda_\std\})$$
such that $f_0=\mbox{id}$ and $f_1(Y \x D^2(R)) \sse Y \x D^2(R_0)$.
\end{thm}

This follows immediately from the (1)$\Longleftrightarrow$(2) equivalence in Theorem \ref{thm:main} together with the $h$--principle for isocontact embeddings into overtwisted manifolds \cite[Corollary 1.4]{BEM}. Theorem \ref{cor: OT squeezing}, being a contact squeezing result, relates to non--orderability \cite{BEM,CP2,EKP,Gi}. The radius $R_0$ in the statement of Theorem \ref{cor: OT squeezing} can be taken to be any radius greater than the minimal radius $R_{c}$ such that the contact manifold
$(Y \x D^2(R_c),\ker\{\alpha + \lambda_\std\})$ is overtwisted. Thus in Theorem \ref{cor: OT squeezing} we can take $R_0$ to be, for instance, twice $R_c$.

In contrast with Theorem \ref{cor: OT squeezing}, there are instances of contact non--squeezing:

\begin{prop}\label{prop:smallsqueez}
Let $(Y,\ker\alpha)$ be a contact $3$-manifold. Then there exists a small radius $\delta\in\R^+$ such that for any $R>\delta$ there exists no contact embedding
$$(Y \x D^2(R),\ker\{\alpha+\lambda_\std\})\longrightarrow(Y \x D^2(\delta),\ker\{\alpha+\lambda_\std\}).$$
\end{prop}

This proposition follows from \cite[Proposition 11]{CPS} and known obstructions to fillability \cite{NP}.

\begin{remark}
Proposition \ref{prop:smallsqueez} also holds in higher--dimensions for any weakly fillable contact structure $(Y^{2n-1},\ker\alpha)$, as it follows by combining F.~Bourgeois' construction \cite{Bo} of contact structures in $Y\times T^2$ and the observation \cite[Example 1.1]{MNW} that the construction preserves weak fillability.\hfill$\Box$
\end{remark}

In addition, we observe that the equivalence (1)$\Longleftrightarrow$(2) shows that contactomorphism type is sensitive to dimensional stabilization.

\begin{coro}\label{cor: OT stab}
There exist closed smooth manifolds $Y$ with two non--isomorphic contact structures $\ker\alpha_1$ and $\ker\alpha_2$ such that $(Y\times\C,\ker\{\alpha_1+\lambda_\std\})$ and $(Y\times\C,\ker\{\alpha_2+\lambda_\std\})$ are contactomorphic.
\end{coro}

\begin{proof}
For instance, we can consider $\xi_1=\ker\alpha_1$ and $\xi_2=\ker\alpha_2$ to be two different overtwisted contact structures on any integral homology $3$-sphere $M$. Then the almost contact structures on the smooth manifold $M$ are classified by homotopy classes of sections of a $SO(3)/U(1)$--bundle over $M^3$, and the obstruction classes thus live in $H^3(M,\Z)$ \cite[Chapter 4.3]{Ha}. The same computation shows that the set of homotopy classes of almost contact structure in the 5--fold $M\times\C$ is determined by the first Chern $c_1\in H^2(M\times\C,\Z)\cong H^2(M,\Z)\cong0$, and thus there exists a unique class of almost contact structures on $M\times\C$. In consequence the two hyperplane fields $\ker\{\alpha_1+\lambda_\std\}$ and $\ker\{\alpha_2+\lambda_\std\}$ become homotopic as almost contact structures in $Y\times\C$, and since both of these contact structures are overtwisted at infinity, they are isotopic contact structures \cite{BEM,El93}.
\end{proof}

Notice that the homotopy class of a compatible almost complex structure structure distinguishes the symplectizations of two different overtwisted contact structures on $S^3$, and thus the symplectizations are not symplectomorphic. Hence Theorem \ref{cor: OT stab} shows that the contactizations of two non--isomorphic complete symplectizations can be contactomorphic.

\subsection{Weinstein cobordisms with overtwisted concave end}\label{ssec:conc}
In this subsection we construct a smooth concordance with a Weinstein structure between an overtwisted contact structure on $S^{2n-1}$ and its standard contact structure $(S^{2n-1},\xi_\std)$ for the higher dimensions $\dim(S^{2n-1})\geq5$. This contrasts with the fact that such a concordance does not exist for $\dim(S^3)=3$ and also provides the general existence result stated in Theorem \ref{thm: weinstein exist}.

\begin{prop}\label{prop: concord}
Suppose that $n\geq3$, then there is a Weinstein structure $(M,\lambda, f)$ on the smoothly trivial cobordism $M \cong [0,1]\x S^{2n-1}$ such that $(\dd_+M,\lambda) \cong (S^{2n-1},\xi_\std)$ and $(\dd_-M,\ker(\lambda))$ is the unique overtwisted contact sphere in the almost contact class of $\xi_\std$.
\end{prop}

\begin{proof}
Let $(S^{2n-1},\xi_\ot)$ be the overtwisted contact sphere in the standard almost contact class and let $M$ be its symplectization. The standard Weinstein structure on $M$ can be homotoped to one with a cancelling pair of critical points, one of index $n-1$ and one of index $n$. Consider a middle contact level $(Y, \xi)$ between these two critical points. Then we can view $Y$ either as a subcritical isotropic surgery on $(S^{2n-1},\xi_\ot)$ induced by the bottom half of the cobordism $M$, or as the result of a $(+1)$--surgery along a Legendrian sphere $\Lambda \sse (S^{2n-1},\xi_\ot)$ induced by the top half of the cobordism $M$. Note that the contact structure $(Y,\xi)$ can be obtained as a subcritical surgery on an overtwisted manifold and thus it is overtwisted. See Figure \ref{fig:concordance} for a schematic picture of the forthcoming argument.

Now let $\Lambda_0 \sse (S^{2n-1},\xi_\std)$ be the loose Legendrian sphere which is in the same formal Legendrian isotopy class as the Legendrian sphere $\Lambda \sse (S^{2n-1},\xi_\ot)$. Note that $(S^{2n-1},\xi_\ot)$ and $(S^{2n-1},\xi_\std)$ are in the same almost contact class, hence the formal Legendrian isotopy classes are canonically identified once we fix a diffeomorphism realizing the almost contact equivalence. Then performing a $(+1)$--surgery along the loose Legendrian $\Lambda_0$ gives a contact manifold $(Y, \xi')$ which is almost contact equivalent to $(Y, \xi)$. Theorem \ref{thm:main} implies that the contact structure $(Y, \xi')$ is overtwisted, and therefore the contact structure $\xi$ is isotopic to $\xi'$.

Let $M_b$ be the bottom half of the Weinstein cobordism $M$ and $M_t$ the Weinstein cobordism from $(Y, \xi')$ to $(S^{2n-1},\xi_\std)$ induced by the above $(+1)$--surgery on $\Lambda_0$. Then the glued cobordism $\wt M := M_b \cup_Y M_t$ is a Weinstein cobordism from $(S^{2n-1},\xi_\ot)$ to $(S^{2n-1},\xi_\std)$ which is diffeomorphic to the smooth concordance $M$.
\end{proof}

\begin{center}
\begin{figure}[h]
\includegraphics[scale=0.7]{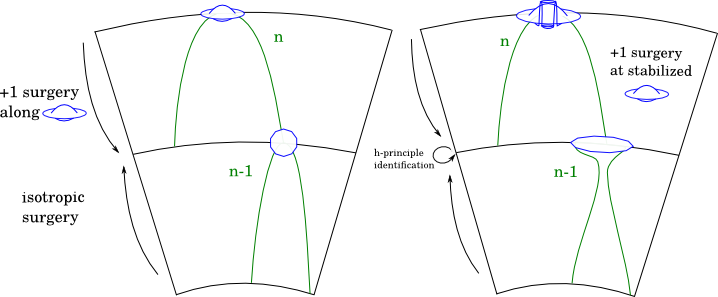}
\caption{On the left, the symplectization of $(S^{2n-1},\xi_{ot})$ with a cancelling pair of critical points. On the right, Weinstein concordance from $(S^{2n-1},\xi_{ot})$ to $(S^{2n-1},\xi_\std)$ obtained using (+1)--surgery on a loose Legendrian.}
\label{fig:concordance}
\end{figure}
\end{center}

Proposition \ref{prop: concord} describes a strictly higher--dimensional phenomenon in contact topology. Indeed, it follows from the functoriality of Seiberg--Witten invariants that there exists no such Weinstein concordance in the case $n=2$ \cite{MR}, \cite[Theorem 2.3]{Hu}, \cite[Chapter 7]{KM}.

The Weinstein concordance among spheres constructed in Proposition \ref{prop: concord} can now be glued to any Weinsten cobordism, thus proving the existence of all Weinstein cobordisms with an overtwisted concave boundary and arbitrary convex end which are not prohibited by topological restrictions:

\begin{thm} \label{thm: weinstein exist}
Let $(Y_-, \xi_\ot)$ and $(Y_+, \xi)$ be coorientable contact manifolds with the contact structure $(Y_-, \xi_\ot)$ being overtwisted and $\dim(Y_-)=\dim(Y_+)\geq5$.

Suppose there exists a smooth cobordism $W$ from $Y_-$ to $Y_+$ such that
\begin{itemize}
 \item[a.] The relative homotopy type of $W$ with respect to its boundary deformation retracts onto a half-dimensional $CW$ complex.
 \item[b.] $W$ admits an almost complex structure $J$ such that the restriction $J|_{Y_+}$, resp.~$J|_{Y_-}$, is homotopic through almost contact structures to $\xi$, resp.~$\xi_\ot$.
\end{itemize}
Then there exists a Weinstein cobordism $(W,\lambda, \p)$ with concave boundary $\dd_-(W, \lambda) = (Y_-, \xi_\ot)$ and convex boundary $\dd_+(W, \lambda) = (Y_+, \xi)$.
\end{thm}

\begin{proof}
Let $(Y_+,\xi_\ot^+)$ be the overtwisted contact structure which is in the same almost contact homotopy class as $\xi$. The existence theorem for flexible Weinstein cobordisms \cite{CE} provides a flexible Weinstein structure $(\lambda_f, \p_f)$ on $W$ such that the almost complex structure $J$ is compatible with the symplectic 2--form $d\lambda_f$ after homotopy, and $\dd_-(W, \lambda_f, \p_f) = (Y_-, \xi_\ot)$. Since the Weinstein cobordism $(\lambda_f, \p_f)$ is flexible by construction it follows that the contact structure $\dd_+(W, \lambda_f)$ in the convex end is overtwisted by using Proposition \ref{prop: flexible ot}. Note also that the contact structure in this convex end $\dd_+(W, \lambda_f)$ is in the same almost contact homotopy class as the initial overtwisted contact structure $\xi_\ot^+$ since both are homotopic to $J|_{Y_+}$. In consequence, we obtain the contactomorphism $\dd_+(W, \lambda_f) \cong (Y_+, \xi_\ot^+)$.

Let us now consider the Weinstein concordance $M = ([0,1]\x S^{2n-1}, \lambda, f)$ constructed in Proposition \ref{prop: concord} and the symplectization $S(Y_+, \xi)$ of the contact structure $(Y_+, \xi)$. Then the connected sum cobordism $M \ol{\#} S(Y_+, \xi)$ is a Weinstein cobordism which is diffeomorphic to the concordance $[0,1]\x Y$ and satisfies
$$\dd_-M \ol{\#} S(Y_+, \xi) \cong (Y_+, \xi_\ot^+),\qquad \dd_+M \ol{\#} S(Y_+, \xi) \cong (Y_+, \xi).$$
Thus, we can concatenate the Weinstein cobordism $(W, \lambda_f, \p_f)$ to this Weinstein smooth concordance $M \ol{\#} S(Y_+, \xi)$ along their common contact boundary
$$\dd_+(W, \lambda_f) \cong (Y_+, \xi_\ot^+) \cong \dd_-M \ol{\#} S(Y_+, \xi)$$
and thus construct the Weinstein cobordism $(W,\la,\p)$ with the desired properties.
\end{proof}

\begin{remark}
Theorem \ref{thm: weinstein exist} is the first step in the proof of the general $h$--principle for symplectic cobordisms with overtwisted concave end proven in \cite{ElMu}.
\end{remark}

\end{document}